\documentclass[preprint,10pt]{elsarticle}

\usepackage{amsmath}
\usepackage{amssymb}
\usepackage{amsthm}

\newtheorem{thm}{Theorem}[section]
\newtheorem{cor}[thm]{Corollary}
\newtheorem{lem}[thm]{Lemma}
\newtheorem{prop}[thm]{Proposition}

\newtheorem{rem}[thm]{Remark}
\numberwithin{equation}{section}

\begin{document}

\begin{frontmatter}



\title{Multilinear fractional maximal and integral operators with homogeneous kernels, Hardy--Littlewood--Sobolev and Olsen-type inequalities}

%% \address[label1]{}[Wang,Hua]

\author{Cong Chen, Kaikai Yang and Hua Wang}
\address{School of Mathematics and System Science, Xinjiang University,\\
Urumqi 830046, P. R. China\\
\textbf{Dedicated to the memory of Li Xue}}
\ead{1362447069@qq.com,wanghua@pku.edu.cn}

\begin{abstract}
Let $m\in \mathbb{N}$ and $0<\alpha<mn$. Let $\mathcal{T}_{\Omega,\alpha;m}$ be the multilinear fractional integral operator with homogeneous kernels defined by
\begin{equation*}
\mathcal{T}_{\Omega,\alpha;m}(\vec{f})(x):=\int_{(\mathbb R^n)^m}\frac{\Omega_1(x-y_1)\cdots\Omega_m(x-y_m)}{\|(x-y_1,\dots,x-y_m)\|^{mn-\alpha}}f_1(y_1)\cdots f_m(y_m)\,dy_1\cdots dy_m,
\end{equation*}
and the related multilinear fractional maximal operator $\mathcal{M}_{\Omega,\alpha;m}$ with homogeneous kernels is given by
\begin{equation*}
\mathcal{M}_{\Omega,\alpha;m}(\vec{f})(x):=\sup_{r>0}m(B(x,r))^{\frac{\alpha}{n}}
\prod_{i=1}^m\frac{1}{m(B(x,r))}\int_{|x-y_i|<r}\big|\Omega_i(x-y_i)f_i(y_i)\big|\,dy_i.
\end{equation*}
In this paper, we will use the idea of Hedberg to reprove that the multilinear operators $\mathcal{T}_{\Omega,\alpha;m}$ and $\mathcal{M}_{\Omega,\alpha;m}$ are bounded from $L^{p_1}(\mathbb R^n)\times L^{p_2}(\mathbb R^n)\times\cdots\times L^{p_m}(\mathbb R^n)$ into $L^q(\mathbb R^n)$ provided that $\vec{\Omega}=(\Omega_1,\Omega_2,\dots,\Omega_m)\in L^s(\mathbf{S}^{n-1})$, $s'<p_1,p_2,\dots,p_m<n/{\alpha}$,
\begin{equation*}
\frac{\,1\,}{p}=\frac{1}{p_1}+\frac{1}{p_2}+\cdots+\frac{1}{p_m} \quad \mbox{and} \quad \frac{\,1\,}{q}=\frac{\,1\,}{p}-\frac{\alpha}{n}.
\qquad (*)
\end{equation*}
This result was first obtained by Chen and Xue. We also prove that under the assumptions that $\vec{\Omega}=(\Omega_1,\Omega_2,\dots,\Omega_m)\in L^s(\mathbf{S}^{n-1})$, $s'\leq p_1,p_2,\dots,p_m<n/{\alpha}$ and $(*)$, the multilinear operators $\mathcal{T}_{\Omega,\alpha;m}$ and $\mathcal{M}_{\Omega,\alpha;m}$ are bounded from $L^{p_1}(\mathbb R^n)\times L^{p_2}(\mathbb R^n)\times \cdots\times L^{p_m}(\mathbb R^n)$ into $L^{q,\infty}(\mathbb R^n)$, which are completely new. Moreover, we will use the idea of Adams to show that $\mathcal{T}_{\Omega,\alpha;m}$ and $\mathcal{M}_{\Omega,\alpha;m}$ are bounded from $L^{p_1,\kappa}(\mathbb R^n)\times L^{p_2,\kappa}(\mathbb R^n)\times \cdots\times L^{p_m,\kappa}(\mathbb R^n)$ into $L^{q,\kappa}(\mathbb R^n)$ whenever $s'<p_1,p_2,\dots,p_m<n/{\alpha}$, $0<\kappa<1$,
\begin{equation*}
\frac{\,1\,}{p}=\frac{1}{p_1}+\frac{1}{p_2}+\cdots+\frac{1}{p_m} \quad \mbox{and} \quad \frac{\,1\,}{q}=\frac{\,1\,}{p}-\frac{\alpha}{n(1-\kappa)},\qquad (**)
\end{equation*}
and also bounded from $L^{p_1,\kappa}(\mathbb R^n)\times L^{p_2,\kappa}(\mathbb R^n)\times \cdots\times L^{p_m,\kappa}(\mathbb R^n)$ into $WL^{q,\kappa}(\mathbb R^n)$ whenever $s'\leq p_1,p_2,\dots,p_m<n/{\alpha}$, $0<\kappa<1$ and $(**)$. Some new estimates in the limiting cases are also established. Applications to the Hardy--Littlewood--Sobolev and Olsen-type inequalities are discussed as well.
\end{abstract}


\end{frontmatter}

\section{Introduction}

\subsection{Linear fractional integrals}
\label{sec11}
In this paper, the symbols $\mathbb R$ and $\mathbb N$ stand for the sets of all real numbers and natural numbers, respectively. Let $\mathbb R^n$ be the $n$-dimensional Euclidean space endowed with the Lebesgue measure $dx$. The Euclidean norm of $x=(x_1,x_2,\dots,x_n)\in \mathbb R^n$ is given by $|x|:=(\sum_{i=1}^n x_i^2)^{1/2}$. Let $\mathbf{S}^{n-1}:=\{x\in\mathbb R^n:|x|=1\}$ denote the unit sphere in $\mathbb R^n$ ($n\geq2$) equipped with the normalized Lebesgue measure $d\sigma(x')$, where we use $x'$ to denote the unit vector in the direction of $x$. Let $\Omega$ be a homogeneous function of degree zero on $\mathbb R^n$ and $\Omega\in L^s(\mathbf{S}^{n-1})$ with $s\geq1$. That is to say:
\begin{enumerate}
  \item for any $\lambda>0$ and $x\in\mathbb R^n$, $\Omega(\lambda x)=\Omega(x)$;
  \item $\|\Omega\|_{L^s(\mathbf{S}^{n-1})}:=\big(\int_{\mathbf{S}^{n-1}}|\Omega(x')|^sd\sigma(x')\big)^{1/s}<+\infty$.
\end{enumerate}
For $\alpha\in(0,n)$, then the homogeneous fractional integral operator $T_{\Omega,\alpha}$ is defined by
\begin{equation*}
T_{\Omega,\alpha}f(x):=\int_{\mathbb R^n}\frac{\Omega(x-y)}{|x-y|^{n-\alpha}}f(y)\,dy,\quad \forall~x\in\mathbb R^n.
\end{equation*}
Given any $\alpha\in[0,n)$ and $f\in L^1_{\mathrm{loc}}(\mathbb R^n)$, the fractional maximal function of $f$ is defined by
\begin{equation*}
M_{\Omega,\alpha}f(x):=\sup_{r>0}\frac{1}{m(B(x,r))^{1-\alpha/n}}\int_{|x-y|<r}\big|\Omega(x-y)f(y)\big|\,dy,\quad\forall~x\in\mathbb R^n,
\end{equation*}
where the supremum on the right-hand side is taken over all $r>0$.
\begin{itemize}
  \item When $\alpha=0$ and $\Omega\equiv1$, $M_{\Omega,\alpha}$ is just the standard Hardy--Littlewood maximal operator on $\mathbb R^n$, which is denoted by $M$;
  \item When $\Omega\equiv1$, $M_{\Omega,\alpha}$ is just the fractional maximal operator $M_{\alpha}$ on $\mathbb R^n$, and $T_{\Omega,\alpha}$ is just the Riesz potential operator $I_{\alpha}$ with $0<\alpha<n$.
\end{itemize}
The Riesz potential can be seen as the fundamental solution of the fractional Laplacian operator. It is well known that the Riesz potential operator $I_{\alpha}$ plays an important role in harmonic analysis, PDE and potential theory, particularly in the study of smoothness properties of functions. The boundedness properties of $I_{\alpha}$ between various Banach spaces have been extensively studied. Let $1<p<q<\infty$. Let us recall the classical Hardy--Littlewood--Sobolev theorem, which states that $I_{\alpha}$ is bounded from $L^p(\mathbb R^n)$ into $L^q(\mathbb R^n)$ if and only if $1/q=1/p-\alpha/n$. We also know that for $1=p<q<\infty$, $I_{\alpha}$ is bounded from $L^1(\mathbb R^n)$ into $L^{q,\infty}(\mathbb R^n)$ if and only if $q=n/{(n-\alpha)}$. However, for the critical index $p=n/{\alpha}$ and $0<\alpha<n$, the strong type $(p,\infty)$ estimate of the operator $I_{\alpha}$ does not hold in general (see \cite[p.119]{stein}). Instead, in this case, we have that as a substitute the Riesz potential operator $I_{\alpha}$ is bounded from $L^{n/{\alpha}}(\mathbb R^n)$ to $\mathrm{BMO}(\mathbb R^n)$ (see \cite[p.130]{duoand}).

For any $x_0\in\mathbb R^n$ and $r>0$, let $B(x_0,r):=\{x\in\mathbb R^n:|x-x_0|<r\}$ denote the open ball centered at $x_0$ of radius $r$, and $B(x_0,r)^{\complement}$ denote its complement. We use the notation $m(B(x_0,r))$ for the Lebesgue measure of the ball $B(x_0,r)$. Given $B=B(x_0,r)$ and $t>0$, we will write $tB$ for the $t$-dilate ball, which is the ball with the same center $x_0$ and with radius $tr$. It can be shown that $M_{\alpha}$ satisfies the same norm inequalities as $I_{\alpha}$ for $0<\alpha<n$, $1<p<n/{\alpha}$ and $1/q=1/p-\alpha/n$. The weak type $(1,q)$ inequality can be proved by using covering lemma arguments when $q=n/{(n-\alpha)}$ (see, for example, \cite[Theorem 2]{muckenhoupt2}), and $M_{\alpha}$ is clearly of strong type $(p,\infty)$ when $p=n/{\alpha}$. Actually, by H\"{o}lder's inequality
\begin{equation*}
\frac{1}{m(B(x,r))^{1-\alpha/n}}\int_{|x-y|<r}|f(y)|\,dy\leq\bigg(\int_{|x-y|<r}|f(y)|^p\,dy\bigg)^{1/p}\leq\|f\|_{L^p}.
\end{equation*}
The strong type $(p,q)$ inequalities are established via the Marcinkiewicz interpolation theorem.

\subsection{Lebesgue and Morrey spaces}
\label{sec12}
Recall that, for any given $p\in(0,\infty)$, the Lebesgue space $L^p(\mathbb R^n)$ is defined as the set of all integrable functions $f$ on $\mathbb R^n$ such that
\begin{equation*}
\|f\|_{L^p}:=\bigg(\int_{\mathbb R^n}|f(x)|^p\,dx\bigg)^{1/p}<+\infty,
\end{equation*}
and the weak Lebesgue space $L^{p,\infty}(\mathbb R^n)$ is defined to be the set of all Lebesgue measurable functions $f$ on $\mathbb R^n$ such that
\begin{equation*}
\|f\|_{L^{p,\infty}}:=\sup_{\lambda>0}\lambda\cdot m\big(\big\{x\in\mathbb R^n:|f(x)|>\lambda\big\}\big)^{1/p}<+\infty.
\end{equation*}
Let $L^{\infty}(\mathbb R^n)$ denote the Banach space of all essentially bounded measurable functions $f$ on $\mathbb R^n$.

On the other hand, the classical Morrey spaces $L^{p,\kappa}(\mathbb R^n)$ were originally introduced by Morrey in \cite{morrey} to study the local regularity of solutions to second order elliptic partial differential equations. Nowadays these spaces have been studied intensively in the literature, and found many important applications in harmonic analysis, potential theory and nonlinear dispersive equations(see \cite{adams1}). Let us now recall the definition of the Morrey space and weak Morrey space. Let $0<p<\infty$ and $0\leq\kappa\leq1$. We denote by $L^{p,\kappa}(\mathbb R^n)$ the Morrey space of all $p$-locally integrable functions $f$ on $\mathbb R^n$ such that
\begin{equation*}
\begin{split}
\|f\|_{L^{p,\kappa}}:=&\sup_{B\subset\mathbb R^n}\bigg(\frac{1}{m(B)^{\kappa}}\int_B|f(x)|^p\,dx\bigg)^{1/p}\\
=&\sup_{B\subset\mathbb R^n}\frac{1}{m(B)^{\kappa/p}}\big\|f\cdot\chi_{B}\big\|_{L^p}<+\infty.
\end{split}
\end{equation*}
Note that $L^{p,0}(\mathbb R^n)=L^p(\mathbb R^n)$ and $L^{p,1}(\mathbb R^n)=L^{\infty}(\mathbb R^n)$ by the Lebesgue
differentiation theorem. If $\kappa<0$ or $\kappa>1$, then $L^{p,\kappa}=\Theta$, where $\Theta$ is the set of all functions equivalent to $0$ on $\mathbb R^n$. We also denote by $WL^{p,\kappa}(\mathbb R^n)$ the weak Morrey space of all measurable functions $f$ on $\mathbb R^n$ such that
\begin{equation*}
\begin{split}
\|f\|_{WL^{p,\kappa}}:=&\sup_{B\subset\mathbb R^n}\frac{1}{m(B)^{\kappa/p}}\big\|f\cdot\chi_{B}\big\|_{L^{p,\infty}}\\
=&\sup_{B\subset\mathbb R^n}\sup_{\lambda>0}\frac{1}{m(B)^{\kappa/p}}\lambda\cdot m\big(\big\{x\in B:|f(x)|>\lambda\big\}\big)^{1/p}<+\infty,
\end{split}
\end{equation*}
where $\chi_{B}$ denotes the characteristic function of the ball $B$: $\chi_B(x)=1$ if $x\in B$ and $0$ if $x\notin B$.

\subsection{Multilinear fractional integrals}
\label{sec13}
Let $2\leq m\in\mathbb N$ and $(\mathbb R^n)^m=\overbrace{\mathbb R^n\times\cdots\times\mathbb R^n}^m$ be the $m$-fold product space. Let $\vec{f}:=(f_1,f_2,\dots,f_m)$ and $0<\alpha<mn$. The multilinear fractional integral operator $\mathcal{I}_{\alpha;m}$ is defined by
\begin{equation*}
\mathcal{I}_{\alpha;m}(\vec{f})(x):=
\int_{(\mathbb R^n)^m}\frac{f_1(y_1)\cdots f_m(y_m)}{\|(x-y_1,\dots,x-y_m)\|^{mn-\alpha}}\,dy_1\cdots dy_m,\quad\forall~x\in\mathbb R^n,
\end{equation*}
where we denote
\begin{equation*}
\big\|(x-y_1,\dots,x-y_m)\big\|:=\Big(\sum_{j=1}^m|x-y_j|^2\Big)^{1/2}.
\end{equation*}
The multilinear fractional integral operator $\mathcal{I}_{\alpha;m}$ is a natural generalization of the linear case. Suppose that each $f_i$ is locally integrable on $\mathbb R^n$, $i=1,2,\dots,m$. Then we define the multilinear fractional maximal operator by
\begin{equation*}
\mathcal{M}_{\alpha;m}(\vec{f})(x):=\sup_{r>0}m(B(x,r))^{\frac{\alpha}{n}}
\prod_{i=1}^m\frac{1}{m(B(x,r))}\int_{|x-y_i|<r}\big|f_i(y_i)\big|\,dy_i,\quad\forall~x\in\mathbb R^n,
\end{equation*}
where the supremum is taken over all $r>0$.
\begin{itemize}
  \item Obviously, when $\alpha=0$, $\mathcal{M}_{\alpha;m}$ is just the new multi(sub)linear maximal function $\mathcal{M}$, which was first defined and studied by Lerner et al. \cite{lerner} in 2009. This new maximal operator can be used to establish the characterization of the multiple $A_{\vec{P}}$ weights and to obtain some weighted estimates for the multilinear Calder\'on--Zygmund singular integral operators and their commutators, see \cite{lerner} for further details.
  \item The multilinear fractional integral $\mathcal{I}_{\alpha;m}$ was introduced and studied by Kenig and Stein \cite{kenig} in 1999. Another type of multilinear fractional integral was considered by Grafakos \cite{gra1} in 1992, and further studied by Grafakos and Kalton \cite{gra2} in 2001. Their works originate from the bilinear fractional integral operator
      \begin{equation*}
      \mathcal{B}_{\alpha}(f,g)(x)=\int_{\mathbb R^n}\frac{f(x+t)g(x-t)}{|t|^{n-\alpha}}\,dt.
      \end{equation*}
      They showed that $\mathcal{B}_{\alpha}$ is bounded from $L^{p_1}(\mathbb R^n)\times L^{p_2}(\mathbb R^n)$ into $L^q(\mathbb R^n)$, where $1/q=1/{p_1}+1/{p_2}-\alpha/n$ with $p_1,p_2>1$.
\end{itemize}

In general, the following strong type and weak type estimates were first given by Kenig and Stein in \cite[Theorem 1]{kenig}(in the special case), see also \cite{chen} and \cite{moen} for the unweighted case.

\begin{thm}[\cite{kenig}]\label{kenigstein}
Let $2\leq m\in \mathbb{N}$ and $0<\alpha<mn$. Assume that $1\leq p_1,p_2,\dots,p_m<\infty$, $1/p=1/{p_1}+1/{p_2}+\cdots+1/{p_m}$ and $1/q=1/p-{\alpha}/n$.

$(i)$ If $p_i>1$, $i=1,2,\dots,m$, then there is a constant $C>0$ independent of $\vec{f}$ such that
\begin{equation*}
\big\|\mathcal{I}_{\alpha;m}(\vec{f})\big\|_{L^q}\leq C\prod_{i=1}^m\|f_i\|_{L^{p_i}}
\end{equation*}
holds for all $\vec{f}=(f_1,f_2,\dots,f_m)\in L^{p_1}(\mathbb R^n)\times L^{p_2}(\mathbb R^n)\times \cdots\times L^{p_m}(\mathbb R^n)$.

$(ii)$ If  at least one $p_i$ equals one, then there is a constant $C>0$ independent of $\vec{f}$ such that
\begin{equation*}
\big\|\mathcal{I}_{\alpha;m}(\vec{f})\big\|_{L^{q,\infty}}\leq C\prod_{i=1}^m\|f_i\|_{L^{p_i}}
\end{equation*}
holds for all $\vec{f}=(f_1,f_2,\dots,f_m)\in L^{p_1}(\mathbb R^n)\times L^{p_2}(\mathbb R^n)\times \cdots\times L^{p_m}(\mathbb R^n)$.
\end{thm}
For simplicity, we denote
\begin{equation*}
|\vec{f}|:=(|f_1|,|f_2|,\dots,|f_m|)\quad \&\quad {(\vec{f})}^{s}:=(f_1^s,f_2^s,\dots,f_m^s),\quad s\geq1.
\end{equation*}
A simple calculation shows that $\mathcal{M}_{\alpha;m}(\vec{f})$ can be controlled pointwise by $\mathcal{I}_{\alpha;m}(|\vec{f}|)$ for any vector-valued function $\vec{f}$. Indeed, for any $0<\alpha<mn$, $x\in\mathbb R^n$ and $r>0$, we have
\begin{equation*}
\begin{split}
\mathcal{I}_{\alpha;m}(|\vec{f}|)(x)&\geq
\int_{|x-y_1|<r}\cdots\int_{|x-y_m|<r}\frac{|f_1(y_1)|\cdots |f_m(y_m)|}{\|(x-y_1,\dots,x-y_m)\|^{mn-\alpha}}\,dy_1\cdots dy_m\\
&\geq\frac{1}{(mr)^{mn-\alpha}}\prod_{i=1}^m\int_{|x-y_i|<r}\big|f_i(y_i)\big|\,dy_i\\
&=C_{m,n,\alpha}\prod_{i=1}^m\frac{1}{m(B(x,r))^{1-\alpha/{mn}}}\int_{|x-y_i|<r}\big|f_i(y_i)\big|\,dy_i,
\end{split}
\end{equation*}
where we have used the estimate below.
\begin{equation*}
\big\|(x-y_1,\dots,x-y_m)\big\|=\Big(\sum_{j=1}^m|x-y_j|^2\Big)^{1/2}\leq\sum_{j=1}^m|x-y_j|<mr.
\end{equation*}
Taking the supremum for all $r>0$ on both sides of the above inequality, we get
\begin{equation}\label{pointwise}
C_{m,n,\alpha}\cdot \mathcal{M}_{\alpha;m}(\vec{f})(x)\leq \mathcal{I}_{\alpha;m}(|\vec{f}|)(x),\quad \forall\; x\in\mathbb R^n.
\end{equation}
Thus the desired result follows immediately (However, the converse inequality does not hold in general, as in the linear case). In view of the above estimate, we can see that the conclusions of Theorem \ref{kenigstein} also hold for the multilinear fractional maximal operator $\mathcal{M}_{\alpha;m}$. Concerning the weighted norm inequalities for multilinear fractional type operators, including the multilinear fractional maximal operator $\mathcal{M}_{\alpha;m}$, the multilinear fractional integral $\mathcal{I}_{\alpha;m}$ and their commutators, see e.g. \cite{chen}, \cite{moen}, \cite{lee} and \cite{pradolini}.

Assume each $\Omega_i$ is a homogeneous function of degree zero on $\mathbb R^n$, i.e., $\Omega_i(\lambda x)=\Omega_i(x)$ for any $\lambda>0$ and $x\in \mathbb R^n$, and $\Omega_i\in L^{s}({\bf S}^{n-1})$($i=1,2,\dots,m$) for some $s\geq1$. Then for any $x\in \mathbb R^n$ and $0<\alpha<mn$, we define the multilinear fractional integral with homogeneous kernels as
\begin{equation*}
\mathcal{T}_{\Omega,\alpha;m}(\vec{f})(x):=
\int_{(\mathbb R^n)^m}\frac{\Omega_1(x-y_1)\cdots\Omega_m(x-y_m)}{\|(x-y_1,\dots,x-y_m)\|^{mn-\alpha}}f_1(y_1)\cdots f_m(y_m)\,dy_1\cdots dy_m,
\end{equation*}
and the corresponding multilinear fractional maximal operator given by
\begin{equation*}
\mathcal{M}_{\Omega,\alpha;m}(\vec{f})(x):=\sup_{r>0}\prod_{i=1}^m
\frac{1}{m(B(x,r))^{1-\alpha/{mn}}}\int_{|x-y_i|<r}\big|\Omega_i(x-y_i)f_i(y_i)\big|\,dy_i.
\end{equation*}
When $\Omega_i$($i=1,2,\dots,m$) satisfy the above condition, we simply write
\begin{equation*}
\vec{\Omega}:=(\Omega_1,\Omega_2,\dots,\Omega_m)\in L^s(\mathbf{S}^{n-1}),\quad s\geq1.
\end{equation*}

In this paper, we will study the boundedness properties of multilinear fractional integrals and fractional maximal
operators with homogeneous kernels on Lebesgue and Morrey spaces. Recall that the following strong type and weak type estimates for $\mathcal{M}_{\Omega,\alpha;m}$ were obtained by Chen and Xue (see the unweighted version of Theorem 2.5 in \cite{chen}).

\begin{thm}[\cite{chen}]
Let $2\leq m\in \mathbb{N}$ and $0<\alpha<mn$, $\vec{\Omega}=(\Omega_1,\Omega_2,\dots,\Omega_m)\in L^s(\mathbf{S}^{n-1})$, $1\leq s'\leq p_1,p_2,\dots,p_m<\infty$,
\begin{equation*}
\frac{\,1\,}{p}=\frac{1}{p_1}+\frac{1}{p_2}+\cdots+\frac{1}{p_m}\quad \&\quad \frac{\,1\,}{q}=\frac{\,1\,}{p}-\frac{\alpha}{\,n\,}.
\end{equation*}
$(i)$ If $p_i>s'$, $i=1,2,\dots,m$, then there is a constant $C>0$ independent of $\vec{f}$ such that
\begin{equation*}
\big\|\mathcal{M}_{\Omega,\alpha;m}(\vec{f})\big\|_{L^q}\leq C\prod_{i=1}^m\|f_i\|_{L^{p_i}}.
\end{equation*}

$(ii)$ If at least one $p_i$ equals $s'$, then there is a constant $C>0$ independent of $\vec{f}$ such that
\begin{equation*}
\big\|\mathcal{M}_{\Omega,\alpha;m}(\vec{f})\big\|_{L^{q,\infty}}\leq C\prod_{i=1}^m\|f_i\|_{L^{p_i}}.
\end{equation*}
\end{thm}
Based on this result, Chen and Xue further obtained the strong type estimate for $\mathcal{T}_{\Omega,\alpha;m}$(see the unweighted version of Theorem 2.6 in \cite{chen}).

\begin{thm}[\cite{chen}]\label{chen1}
Let $2\leq m\in \mathbb{N}$ and $0<\alpha<mn$, $\vec{\Omega}=(\Omega_1,\Omega_2,\dots,\Omega_m)\in L^s(\mathbf{S}^{n-1})$, $1\leq s'<p_1,p_2,\dots,p_m<\infty$,
\begin{equation*}
\frac{\,1\,}{p}=\frac{1}{p_1}+\frac{1}{p_2}+\cdots+\frac{1}{p_m} \quad \& \quad \frac{\,1\,}{q}=\frac{\,1\,}{p}-\frac{\alpha}{\,n\,}.
\end{equation*}
Then there is a constant $C>0$ independent of $\vec{f}$ such that
\begin{equation*}
\big\|\mathcal{T}_{\Omega,\alpha;m}(\vec{f})\big\|_{L^q}\leq C\prod_{i=1}^m\|f_i\|_{L^{p_i}}.
\end{equation*}
\end{thm}

We will give an alternative proof of Theorem \ref{chen1}, following the basic idea due to Hedberg in \cite{hed}. We will establish the corresponding weak type estimate for $\mathcal{T}_{\Omega,\alpha;m}$ when $\vec{\Omega}=(\Omega_1,\Omega_2,\dots,\Omega_m)\in L^s(\mathbf{S}^{n-1})$ with $s'\leq p_1,p_2,\dots,p_m<\infty$ and $1/q=1/{p_1}+1/{p_2}+\cdots+1/{p_m}-{\alpha}/n$. We can prove the same norm inequalities involving multilinear fractional maximal operator $\mathcal{M}_{\Omega,\alpha;m}$. Besides Lebesgue spaces, the boundedness of both operators $\mathcal{T}_{\Omega,\alpha;m}$ and $\mathcal{M}_{\Omega,\alpha;m}$ on Morrey spaces are obtained. Moreover, some new estimates in the limiting cases are also discussed. Finally, we will apply our main results to several well-known inequalities on the Euclidean space $\mathbb R^n$ such as Hardy--Littlewood--Sobolev and Olsen-type inequalities.

\begin{rem}
For the linear case $m=1$, a few historical remarks are given below.
\begin{itemize}
  \item In \cite{muckenhoupt1}, Muckenhoupt and Wheeden proved that $T_{\Omega,\alpha}=\mathcal{T}_{\Omega,\alpha;1}$ is bounded from $L^p(\mathbb R^n)$ into $L^q(\mathbb R^n)$, when $\Omega\in L^s(\mathbf{S}^{n-1})$ with $s>p'$, $1<p<n/{\alpha}$ and $1/q=1/p-{\alpha}/n$ (see also \cite{ding1,lu} for the weighted version).
  \item In \cite{cha},  Chanillo, Watson and Wheeden showed that $T_{\Omega,\alpha}$ is of weak type $(1,n/{(n-\alpha)})$ when $\Omega\in L^s(\mathbf{S}^{n-1})$ with $s\geq n/{(n-\alpha)}$.
  \item Moreover, in \cite{ding2}, Ding and Lu considered the critical case $p=n/{\alpha}$, and proved that if $\Omega$ satisfies certain smoothness condition $\mathcal{D}_{s}$ with $s>n/{(n-\alpha)}$, then $T_{\Omega,\alpha}$ is bounded from $L^{p}(\mathbb R^n)$ into $\mathrm{BMO}(\mathbb R^n)$. For the critical index $p=n/{\alpha}$, we can show that this result also holds for $s=n/{(n-\alpha)}$(see \cite{wang2}).
\end{itemize}
\end{rem}

Throughout this paper, the letter $C$ will stand for a positive constant, not necessarily the same at each occurrence but independent of the essential variables. The symbol $\mathbf{X}\lesssim \mathbf{Y}$ means that there is a constant $C>0$ such that $\mathbf{X}\leq C\mathbf{Y}$.
We will also use the notation $p'=p/{(p-1)}$ for $1<p<\infty$, and $p'=1$ for $p=\infty$.

\section{Boundedness of $\mathcal{T}_{\Omega,\alpha;m}$ and $\mathcal{M}_{\Omega,\alpha;m}$ on Lebesgue spaces}
\label{sec2}
In this section, we will establish boundedness properties of multilinear fractional integral operator and fractional maximal operator on Lebesgue spaces. As pointed out in Section \ref{sec13}, the multi(sub)linear maximal function $\mathcal{M}$ is defined by
\begin{equation}\label{wangM}
\mathcal{M}(\vec{f})(x):=\sup_{r>0}
\prod_{i=1}^m\frac{1}{m(B(x,r))}\int_{B(x,r)}\big|f_i(y_i)\big|\,dy_i,\quad\forall~x\in\mathbb R^n.
\end{equation}
It can be shown that this operator is strictly smaller than the $m$-fold product of $M$ (see \cite{lerner}), where $M$ is the usual Hardy--Littlewood maximal operator.
In addition, it is obvious that for any $\tau>0$,
\begin{equation}\label{prod1}
\prod_{i=1}^m\big(A_i\big)^{\tau}=\Big(\prod_{i=1}^m A_i\Big)^{\tau}.
\end{equation}

\begin{thm}\label{thm1}
Let $\vec{\Omega}=(\Omega_1,\Omega_2,\dots,\Omega_m)\in L^s(\mathbf{S}^{n-1})$ with $1<s\leq\infty$. Let $0<\alpha<mn$, $s'\leq p_1,p_2,\dots,p_m<\infty$,
\begin{equation*}
1/p=1/{p_1}+1/{p_2}+\cdots+1/{p_m}\quad \mathrm{and}\quad 1/q=1/p-\alpha/n.
\end{equation*}

$(i)$ If $s'<p_1,p_2,\dots,p_m<\infty$, then for every $p\in(s'/m,n/{\alpha})$, the following inequality
\begin{equation*}
\big\|\mathcal{T}_{\Omega,\alpha;m}(\vec{f})\big\|_{L^q}\lesssim\prod_{i=1}^m\big\|f_i\big\|_{L^{p_i}}
\end{equation*}
holds for all $\vec{f}=(f_1,f_2,\dots,f_m)\in L^{p_1}(\mathbb R^n)\times L^{p_2}(\mathbb R^n)\times \cdots\times L^{p_m}(\mathbb R^n)$.

$(ii)$ If $s'=\min\{p_1,p_2,\dots,p_m\}$, then for every $p\in[s'/m,n/{\alpha})$, the following inequality
\begin{equation*}
\big\|\mathcal{T}_{\Omega,\alpha;m}(\vec{f})\big\|_{L^{q,\infty}}\lesssim\prod_{i=1}^m\big\|f_i\big\|_{L^{p_i}}
\end{equation*}
holds for all $\vec{f}=(f_1,f_2,\dots,f_m)\in L^{p_1}(\mathbb R^n)\times L^{p_2}(\mathbb R^n)\times \cdots\times L^{p_m}(\mathbb R^n)$.
\end{thm}

Before giving the proof of Theorem \ref{thm1}, we need some facts on the boundedness properties of $\mathcal{M}$ (see \cite{lerner}, for instance).
\begin{lem}[\cite{lerner}]\label{multimax}
Let $1\leq p_1,p_2,\dots,p_m<\infty$, $0<p<\infty$ and $1/p=1/{p_1}+1/{p_2}+\cdots+1/{p_m}$.
\begin{enumerate}
  \item If $1<p_i<\infty$, $i=1,2,\dots,m$, then the inequality
  \begin{equation*}
    \big\|\mathcal{M}(\vec{f})\big\|_{L^p}\leq C\prod_{i=1}^m\big\|f_i\big\|_{L^{p_i}}
  \end{equation*}
  holds for every $\vec{f}=(f_1,f_2,\dots,f_m)\in L^{p_1}(\mathbb R^n)\times L^{p_2}(\mathbb R^n)\times \cdots\times L^{p_m}(\mathbb R^n)$.
  \item If $\min\{p_1,p_2,\dots,p_m\}=1$, then the inequality
\begin{equation*}
\big\|\mathcal{M}(\vec{f})\big\|_{L^{p,\infty}}\leq C\prod_{i=1}^m\big\|f_i\big\|_{L^{p_i}}
\end{equation*}
holds for every $\vec{f}=(f_1,f_2,\dots,f_m)\in L^{p_1}(\mathbb R^n)\times L^{p_2}(\mathbb R^n)\times \cdots\times L^{p_m}(\mathbb R^n)$.
\end{enumerate}
\end{lem}

\begin{proof}[Proof of Theorem \ref{thm1}]
We adopt some basic ideas due to Hedberg in \cite{hed}. For given $\vec{f}=(f_1,\dots,f_m)\in L^{p_1}(\mathbb R^n)\times \cdots\times L^{p_m}(\mathbb R^n)$ with $s'/m\leq p<n/{\alpha}$, we will decomposes the integral defining $\mathcal{T}_{\Omega,\alpha;m}$ into two parts. Using the definition of the multilinear maximal operator $\mathcal{M}$ to the first part and applying the H\"{o}lder inequality to the second. Then we choose a positive number so that both terms are the same size. Arguing as in \cite{hed}, we will prove the following result.
\begin{equation}\label{wanghua1}
\begin{split}
\big|\mathcal{T}_{\Omega,\alpha;m}(\vec{f})(x)\big|\lesssim
\Big[\mathcal{M}_{s'}(\vec{f})(x)\Big]^{p/q}\cdot\Big[\prod_{i=1}^m\|f_i\|_{L^{p_i}}\Big]^{1-p/q},
\quad \forall~x\in\mathbb R^n.
\end{split}
\end{equation}
Let $\vec{f}\in L^{p_1}(\mathbb R^n)\times \cdots\times L^{p_m}(\mathbb R^n)$ with $s'/m\leq p<n/{\alpha}$. In order to prove \eqref{wanghua1}, one can write
\begin{equation*}
\begin{split}
&\big|\mathcal{T}_{\Omega,\alpha;m}(\vec{f})(x)\big|\leq
\int_{(\mathbb R^n)^m}\frac{|\Omega_1(x-y_1)\cdots
\Omega_m(x-y_m)|}{\|(x-y_1,\dots,x-y_m)\|^{mn-\alpha}}
\big|f_1(y_1)\cdots f_m(y_m)\big|\,dy_1\cdots dy_m\\
&=\int_{\|(x-y_1,\dots,x-y_m)\|<\sigma}\frac{|\Omega_1(x-y_1)\cdots\Omega_m(x-y_m)|}{\|(x-y_1,\dots,x-y_m)\|^{mn-\alpha}}
\big|f_1(y_1)\cdots f_m(y_m)\big|\,dy_1\cdots dy_m\\
&+\int_{\|(x-y_1,\dots,x-y_m)\|\geq\sigma}\frac{|\Omega_1(x-y_1)\cdots\Omega_m(x-y_m)|}{\|(x-y_1,\dots,x-y_m)\|^{mn-\alpha}}
\big|f_1(y_1)\cdots f_m(y_m)\big|\,dy_1\cdots dy_m\\
&:=\mathrm{I+II},
\end{split}
\end{equation*}
where $\sigma>0$ is a positive constant to be determined later. To estimate the first term $\mathrm{I}$, for any $x\in\mathbb{R}^n$,
we observe that
\begin{equation}\label{geq}
\big\|(x-y_1,\dots,x-y_m)\big\|=\Big(\sum_{j=1}^m|x-y_j|^2\Big)^{1/2}\geq |x-y_i|,\quad i=1,2,\dots,m.
\end{equation}
Hence,
\begin{equation*}
\begin{split}
\mathrm{I}&=\sum_{j=1}^\infty\int_{2^{-j}\sigma\leq\|(x-y_1,\dots,x-y_m)\|<2^{-j+1}\sigma}
\frac{|\Omega_1(x-y_1)\cdots\Omega_m(x-y_m)|}{\|(x-y_1,\dots,x-y_m)\|^{mn-\alpha}}
\prod_{i=1}^m|f_i(y_i)|\,dy_i\\
&\leq\sum_{j=1}^\infty\frac{1}{(2^{-j}\sigma)^{mn-\alpha}}\int_{2^{-j}\sigma\leq\|(x-y_1,\dots,x-y_m)\|<2^{-j+1}\sigma}
\prod_{i=1}^m|\Omega_i(x-y_i)|\cdot|f_i(y_i)|\,dy_i\\
&\leq\sum_{j=1}^\infty\frac{1}{(2^{-j}\sigma)^{mn-\alpha}}\prod_{i=1}^m\int_{|x-y_i|<2^{-j+1}\sigma}|\Omega_i(x-y_i)|\cdot|f_i(y_i)|\,dy_i.
\end{split}
\end{equation*}
For any $\mathcal{R}>0$ and $\Omega\in L^s(\mathbf{S}^{n-1})$, we first establish the following estimate, which will be often used in the sequel.
\begin{equation}\label{omega88}
\bigg(\int_{|x-y_i|<\mathcal{R}}|\Omega(x-y_i)|^s\,dy_i\bigg)^{1/s}
\leq \Big(\frac{\,1\,}{n}\Big)^{1/s}\|\Omega\|_{L^s(\mathbf{S}^{n-1})}\mathcal{R}^{n/s},\quad x\in\mathbb R^n.
\end{equation}
To this end, using polar coordinates, we obtain that for any given $x\in\mathbb R^n$,
\begin{equation*}
\begin{split}
\bigg(\int_{|x-y_i|<\mathcal{R}}|\Omega(x-y_i)|^s\,dy_i\bigg)^{1/s}
&=\bigg(\int_{|y_i|<\mathcal{R}}|\Omega(y_i)|^{s}\,dy_i\bigg)^{1/s}\\
&=\bigg(\int_0^{\mathcal{R}}\int_{\mathbf{S}^{n-1}}|\Omega(y'_i)|^{s}\varrho^{n-1}\,d\sigma(y'_i)d\varrho\bigg)^{1/s}\\
&\leq \Big(\frac{\,1\,}{n}\Big)^{1/s}\|\Omega\|_{L^s(\mathbf{S}^{n-1})}\mathcal{R}^{n/s},
\end{split}
\end{equation*}
as desired. Here $y'_i=y_i/{|y_i|}$ for any $\mathbb R^n\ni y_i\neq0$. A combination of H\"{o}lder's inequality and \eqref{omega88} gives us that for each fixed $j\geq1$,
\begin{equation*}
\begin{split}
&\int_{|x-y_i|<2^{-j+1}\sigma}|\Omega_i(x-y_i)|\cdot|f_i(y_i)|\,dy_i\\
&\leq\bigg(\int_{|x-y_i|<2^{-j+1}\sigma}|\Omega_i(x-y_i)|^s\,dy_i\bigg)^{1/s}
\bigg(\int_{|x-y_i|<2^{-j+1}\sigma}|f_i(y_i)|^{s'}\,dy_i\bigg)^{1/{s'}}\\
&\lesssim\|\Omega_i\|_{L^s(\mathbf{S}^{n-1})}\big(2^{-j+1}\sigma\big)^{n/s}
\bigg(\int_{B(x,2^{-j+1}\sigma)}|f_i(y_i)|^{s'}\,dy_i\bigg)^{1/{s'}}\\
&\lesssim\|\Omega_i\|_{L^s(\mathbf{S}^{n-1})}\big(2^{-j+1}\sigma\big)^{n}
\times\bigg(\frac{1}{m(B(x,2^{-j+1}\sigma))}\int_{B(x,2^{-j+1}\sigma)}|f_i(y_i)|^{s'}\,dy_i\bigg)^{1/{s'}}.
\end{split}
\end{equation*}
Denote
\begin{equation*}
\|\vec{\Omega}\|_{L^s(\mathbf{S}^{n-1})}:=\prod_{i=1}^m\|\Omega_i\|_{L^s(\mathbf{S}^{n-1})}.
\end{equation*}
Note that $\alpha>0$. Hence
\begin{equation*}
\begin{split}
\mathrm{I}&\lesssim\|\vec{\Omega}\|_{L^s(\mathbf{S}^{n-1})}\sum_{j=1}^\infty\frac{1}{(2^{-j}\sigma)^{mn-\alpha}}\prod_{i=1}^m
\big(2^{-j+1}\sigma\big)^{n}\\
&\times\bigg(\frac{1}{m(B(x,2^{-j+1}\sigma))}\int_{B(x,2^{-j+1}\sigma)}|f_i(y_i)|^{s'}\,dy_i\bigg)^{1/{s'}}\\
&=\|\vec{\Omega}\|_{L^s(\mathbf{S}^{n-1})}\sum_{j=1}^\infty\frac{2^{mn}}{(2^{-j}\sigma)^{-\alpha}}
\prod_{i=1}^m\bigg(\frac{1}{m(B(x,2^{-j+1}\sigma))}\int_{B(x,2^{-j+1}\sigma)}|f_i(y_i)|^{s'}\,dy_i\bigg)^{1/{s'}}\\
&\lesssim\sigma^{\alpha}\mathcal{M}_{s'}(\vec{f})(x)\times\sum_{j=1}^\infty\frac{1}{2^{j\alpha}}
\lesssim\sigma^{\alpha}\mathcal{M}_{s'}(\vec{f})(x),
\end{split}
\end{equation*}
where the multilinear maximal operator $\mathcal{M}_{s'}$ is given by
\begin{equation*}
\mathcal{M}_{s'}(\vec{f})(x):=\big[\mathcal{M}(\vec{f})^{s'}(x)\big]^{1/{s'}},\quad x\in\mathbb R^n.
\end{equation*}
Now we turn to deal with the other term $\mathrm{II}$. In view of \eqref{geq}, we can deduce that
\begin{equation*}
\begin{split}
\mathrm{II}&=\sum_{j=1}^\infty\int_{2^{j-1}\sigma\leq\|(x-y_1,\dots,x-y_m)\|<2^j\sigma}
\frac{|\Omega_1(x-y_1)\cdots\Omega_m(x-y_m)|}{\|(x-y_1,\dots,x-y_m)\|^{mn-\alpha}}
\prod_{i=1}^m|f_i(y_i)|\,dy_i\\
&\leq\sum_{j=1}^\infty\frac{1}{(2^{j-1}\sigma)^{mn-\alpha}}\int_{2^{j-1}\sigma\leq\|(x-y_1,\dots,x-y_m)\|<2^j\sigma}
\prod_{i=1}^m|\Omega_i(x-y_i)|\cdot|f_i(y_i)|\,dy_i\\
&\leq\sum_{j=1}^\infty\frac{1}{(2^{j-1}\sigma)^{mn-\alpha}}
\prod_{i=1}^m\int_{|x-y_i|<2^j\sigma}|\Omega_i(x-y_i)|\cdot|f_i(y_i)|\,dy_i.
\end{split}
\end{equation*}
On the other hand, since $\Omega_i\in L^s(\mathbf{S}^{n-1})$ with $s\geq p_i'$, then $\Omega_i\in L^{p_i'}(\mathbf{S}^{n-1})$, and
\begin{equation*}
\|\Omega_i\|_{L^{p_i'}(\mathbf{S}^{n-1})}\leq C\|\Omega_i\|_{L^s(\mathbf{S}^{n-1})}.
\end{equation*}
This estimate, along with H\"{o}lder's inequality and \eqref{omega88}, yields
\begin{equation}\label{maines1}
\begin{split}
&\int_{|x-y_i|<2^j\sigma}|\Omega_i(x-y_i)|\cdot|f_i(y_i)|\,dy_i\\
&\leq\bigg(\int_{|x-y_i|<2^j\sigma}|\Omega_i(x-y_i)|^{p_i'}\,dy_i\bigg)^{1/{p_i'}}
\bigg(\int_{|x-y_i|<2^j\sigma}|f_i(y_i)|^{p_i}\,dy_i\bigg)^{1/{p_i}}\\
&\leq C\|\Omega_i\|_{L^{p_i'}(\mathbf{S}^{n-1})}\big(2^j\sigma\big)^{n/{p_i'}}\|f_i\|_{L^{p_i}}\\
&\leq C\|\Omega_i\|_{L^s(\mathbf{S}^{n-1})}\big(2^j\sigma\big)^{n/{p_i'}}\|f_i\|_{L^{p_i}},
\end{split}
\end{equation}
for each integer $j\geq1$. Therefore, by \eqref{prod1} and \eqref{maines1}, we get
\begin{equation*}
\begin{split}
\mathrm{II}&\lesssim\|\vec{\Omega}\|_{L^s(\mathbf{S}^{n-1})}
\sum_{j=1}^\infty\frac{1}{(2^{j-1}\sigma)^{mn-\alpha}}\cdot\prod_{i=1}^m\big(2^j\sigma\big)^{n/{p_i'}}\|f_i\|_{L^{p_i}}\\
&=\|\vec{\Omega}\|_{L^s(\mathbf{S}^{n-1})}\sum_{j=1}^\infty\frac{1}{(2^{j-1}\sigma)^{mn-\alpha}}\cdot
\big(2^j\sigma\big)^{\sum_{i=1}^m n/{p_i'}}\prod_{i=1}^m\|f_i\|_{L^{p_i}}.
\end{split}
\end{equation*}
A trivial computation leads to that
\begin{equation}\label{equality1}
\sum_{i=1}^m n/{p_i'}=n\sum_{i=1}^m\big(1-1/{p_i}\big)=mn-n/p.
\end{equation}
Consequently
\begin{equation*}
\begin{split}
\mathrm{II}&\lesssim\sum_{j=1}^\infty\frac{1}{(2^{j}\sigma)^{n/p-\alpha}}\prod_{i=1}^m\|f_i\|_{L^{p_i}}\\
&\lesssim \sigma^{\alpha-n/p}\prod_{i=1}^m\|f_i\|_{L^{p_i}},
\end{split}
\end{equation*}
where the last series is convergent since $n/p-\alpha>0$. Summing up the above estimates for $\mathrm{I}$ and $\mathrm{II}$, we conclude that for any $x\in\mathbb R^n$,
\begin{equation}\label{main1}
\big|\mathcal{T}_{\Omega,\alpha;m}(\vec{f})(x)\big|\lesssim\bigg[\sigma^{\alpha}\mathcal{M}_{s'}(\vec{f})(x)
+\sigma^{\alpha-n/p}\prod_{i=1}^m\|f_i\|_{L^{p_i}}\bigg].
\end{equation}
We now select an appropriate real number $\sigma>0$ such that
\begin{equation*}
\sigma^{\alpha}\mathcal{M}_{s'}(\vec{f})(x)=\sigma^{\alpha-n/p}\prod_{i=1}^m\|f_i\|_{L^{p_i}},
\end{equation*}
that is,
\begin{equation*}
\sigma^{n/p}=\frac{\prod_{i=1}^m\|f_i\|_{L^{p_i}}}{\mathcal{M}_{s'}(\vec{f})(x)}.
\end{equation*}
Putting this value of $\sigma$ back into \eqref{main1} and using the fact that
\begin{equation*}
1-{\alpha p}/n=p(1/p-\alpha/n)=p/q,
\end{equation*}
we can see that
\begin{equation*}
\begin{split}
\big|\mathcal{T}_{\Omega,\alpha;m}(\vec{f})(x)\big|
\lesssim\sigma^{\alpha}\mathcal{M}_{s'}(\vec{f})(x)
&=\Big[\frac{\prod_{i=1}^m\|f_i\|_{L^{p_i}}}{\mathcal{M}_{s'}(\vec{f})(x)}\Big]^{{\alpha p}/n}\cdot \mathcal{M}_{s'}(\vec{f})(x)\\
&=\Big[\mathcal{M}_{s'}(\vec{f})(x)\Big]^{p/q}\cdot\Big[\prod_{i=1}^m\|f_i\|_{L^{p_i}}\Big]^{1-p/q}.
\end{split}
\end{equation*}
Thus, \eqref{wanghua1} holds. The conclusion of Theorem \ref{thm1} then follows from \eqref{wanghua1} and the boundedness of $\mathcal{M}$ (Lemma \ref{multimax}).
\begin{itemize}
  \item If $s'<p_1,p_2,\dots,p_m$ and $p<n/{\alpha}$, then we have
\begin{equation*}
\begin{split}
\big\|\mathcal{T}_{\Omega,\alpha;m}(\vec{f})\big\|_{L^q}&\lesssim\bigg(\int_{\mathbb R^n}\big|\mathcal{M}_{s'}(\vec{f})(x)\big|^{p}dx\bigg)^{1/q} \cdot\Big[\prod_{i=1}^m\|f_i\|_{L^{p_i}}\Big]^{1-p/q}\\
&=\bigg(\int_{\mathbb R^n}\big|\mathcal{M}(\vec{f})^{s'}(x)\big|^{p/{s'}}dx\bigg)^{1/q}\cdot\Big[\prod_{i=1}^m\|f_i\|_{L^{p_i}}\Big]^{1-p/q}.
\end{split}
\end{equation*}
Thus, using the strong-type boundedness of $\mathcal{M}$(part (1) of Lemma \ref{multimax}) and \eqref{prod1}, we obtain
\begin{equation*}
\begin{split}
\big\|\mathcal{T}_{\Omega,\alpha;m}(\vec{f})\big\|_{L^q}&\lesssim\bigg[\prod_{i=1}^m\bigg(\int_{\mathbb R^n}|f_i(x)|^{s'\cdot({p_i}/{s'})}dx\bigg)^{(p/{s'})\cdot({s'}/{p_i})}\bigg]^{1/q}
\cdot\Big[\prod_{i=1}^m\|f_i\|_{L^{p_i}}\Big]^{1-p/q}\\
&=\Big[\prod_{i=1}^m\|f_i\|^p_{L^{p_i}}\Big]^{1/q}\cdot\Big[\prod_{i=1}^m\|f_i\|_{L^{p_i}}\Big]^{1-p/q}\\
&=\Big[\prod_{i=1}^m\|f_i\|_{L^{p_i}}\Big]^{p/q}\cdot\Big[\prod_{i=1}^m\|f_i\|_{L^{p_i}}\Big]^{1-p/q}=\prod_{i=1}^m\|f_i\|_{L^{p_i}}.
\end{split}
\end{equation*}
This proves part $(i)$.
  \item If $s'=\min\{p_1,p_2,\dots,p_m\}$ and $p<n/{\alpha}$, then for any given $\lambda>0$, we have
\begin{equation*}
\begin{split}
&\lambda\cdot m\big(\big\{x\in\mathbb R^n:\big|\mathcal{T}_{\Omega,\alpha;m}(\vec{f})(x)\big|>\lambda\big\}\big)^{1/q}\\
&=\lambda\cdot m\bigg(\bigg\{x\in\mathbb R^n:\big|\mathcal{M}_{s'}(\vec{f})(x)\big|^{p/q}>\frac{\lambda}{C\big[\prod_{i=1}^m\|f_i\|_{L^{p_i}}\big]^{1-p/q}}\bigg\}\bigg)^{1/q}\\
&=\lambda\cdot m\bigg(\bigg\{x\in\mathbb R^n:\big|\mathcal{M}(\vec{f})^{s'}(x)\big|>
\bigg(\frac{\lambda}{C\big[\prod_{i=1}^m\|f_i\|_{L^{p_i}}\big]^{1-p/q}}\bigg)^{{(s'q)}/p}\bigg\}\bigg)^{1/q}.
\end{split}
\end{equation*}
By using the weak-type boundedness of $\mathcal{M}$(part (2) of Lemma \ref{multimax}) and \eqref{prod1}, the above expression is bounded by
\begin{equation*}
\begin{split}
&C\lambda\cdot\bigg[\bigg(\frac{\big[\prod_{i=1}^m\|f_i\|_{L^{p_i}}\big]^{1-p/q}}{\lambda}\bigg)^{{(s'q)}/p}
\cdot\prod_{i=1}^m\bigg(\int_{\mathbb R^n}|f_i(x)|^{s'\cdot({p_i}/{s'})}dx\bigg)^{{s'}/{p_i}}\bigg]^{p/{(s'q)}}\\
&=C\Big[\prod_{i=1}^m\|f_i\|_{L^{p_i}}\Big]^{1-p/q}\cdot\Big[\prod_{i=1}^m\|f_i\|^{s'}_{L^{p_i}}\Big]^{p/{(s'q)}}\\
&=C\Big[\prod_{i=1}^m\|f_i\|_{L^{p_i}}\Big]^{1-p/q}\cdot\Big[\prod_{i=1}^m\|f_i\|_{L^{p_i}}\Big]^{p/q}=C\prod_{i=1}^m\|f_i\|_{L^{p_i}}.
\end{split}
\end{equation*}
This proves part $(ii)$ by taking the supremum over all $\lambda>0$.
\end{itemize}
The proof of Theorem \ref{thm1} is now complete.
\end{proof}

It is well known that the Hardy--Littlewood maximal operator $M$ is of weak $(1,1)$-type and strong $(p,p)$-type, $1<p\leq\infty$. The strong $(1,1)$ inequality is false. In fact, it never holds, as the following result shows (see \cite[p.36]{duoand}).
\begin{prop}
If $f\in L^1(\mathbb R^n)$ and is not identically $0$, then $Mf\notin L^1(\mathbb R^n)$.
\end{prop}
For any given ball $B$ in $\mathbb R^n$, we say that $f\log^{+}f\in L^1(B)$, if
\begin{equation*}
\int_{B}|f(x)|\log^+|f(x)|\,dx<+\infty,
\end{equation*}
where $\log^+t=\max\{\log t,0\}$. Nevertheless, we do have the following result regarding local integrability of the maximal function.
\begin{prop}
Let $f$ be an integrable function supported in a ball $B\subset\mathbb R^n$. Then $Mf\in L^1(B)$ if and only if $f\log^{+}f\in L^1(B)$.
\end{prop}
This result is due to Stein \cite{stein2} (see also \cite[p.42]{duoand} and \cite[p.23]{stein}). In order to deal with the endpoint case, we need to introduce the space $L\log L(B)$, which is defined by
\begin{equation*}
L\log L(B):=\bigg\{f~\mathrm{is~supported~in}~B:\int_{B}|f(x)|\big(1+\log^+|f(x)|\big)\,dx<+\infty\bigg\}.
\end{equation*}
The class $L\log L$ was originally studied by Stein \cite{stein2}. Unfortunately, the expression $\int_{B}|f(x)|\big(1+\log^+|f(x)|\big)\,dx$ does not define a norm. Given a ball $B\subset\mathbb R^n$, we define the following Luxemburg norm:
\begin{equation*}
\|f\|_{L\log L(B)}:=\inf\bigg\{\lambda>0:\int_{B}\frac{|f(x)|}{\lambda}\bigg(1+\log^{+}\Big(\frac{|f(x)|}{\lambda}\Big)\bigg)dx\leq1\bigg\}.
\end{equation*}
Then $L\log L(B)$ becomes a Banach function space (referred to as Orlicz space, which is a generalization of the $L^p$ space) with respect to the norm $\|\cdot\|_{L\log L(B)}$. We rely on the following result, as the author pointed out in \cite[p.42]{duoand}.
\begin{equation}\label{Mlog}
\|Mf\|_{L^1(B)}\leq C\|f\|_{L\log L(B)}.
\end{equation}
More generally, for $1\leq p<\infty$, we now define the space $L^p\log L(B)$ as follows.
\begin{equation*}
L^p\log L(B):=\bigg\{f~\mathrm{is~supported~in}~B:\int_{B}|f(x)|^p\big(1+\log^+|f(x)|\big)\,dx<+\infty\bigg\}.
\end{equation*}
Correspondingly, the Luxemburg norm of $f$ is defined by
\begin{equation*}
\|f\|_{L^p\log L(B)}:=\inf\bigg\{\lambda>0:\int_{B}\bigg(\frac{|f(x)|}{\lambda}\bigg)^p\bigg(1+\log^{+}\Big(\frac{|f(x)|}{\lambda}\Big)\bigg)dx\leq1\bigg\}.
\end{equation*}
By definition, it is obvious that for each given $p$, $L^p\log L(B)\subset L^p(B)$, and
\begin{equation}\label{lplogl}
\|f\|_{L^p(B)}\leq \|f\|_{L^p\log L(B)}.
\end{equation}
In addition, it is easy to verify that for any $1\leq p<\infty$,
\begin{equation}\label{pp}
\big\||f|^p\big\|_{L\log L(B)}\leq \|f\|^p_{L^p\log L(B)}.
\end{equation}
It follows immediately from \eqref{Mlog} and \eqref{pp} that
\begin{equation}\label{Mpp}
\|M(|f|^p)\|_{L^1(B)}\leq C\|f\|^p_{L^p\log L(B)}.
\end{equation}

Based on the above estimates, we are going to prove the following result for the endpoint case $\min\{p_1,p_2,\dots,p_m\}=s'$.

\begin{thm}\label{thm7}
Let $\vec{\Omega}=(\Omega_1,\Omega_2,\dots,\Omega_m)\in L^s(\mathbf{S}^{n-1})$ with $1<s\leq\infty$. Let $0<\alpha<mn$,
\begin{equation*}
1/p=1/{p_1}+1/{p_2}+\cdots+1/{p_m}\quad \mathrm{and}\quad 1/q=1/p-\alpha/n.
\end{equation*}
If $s'=\min\{p_1,p_2,\dots,p_m\}$ and $p<n/{\alpha}$, then for any given ball $B\subset\mathbb R^n$, the inequality
\begin{equation*}
\big\|\mathcal{T}_{\Omega,\alpha;m}(\vec{f})\big\|_{L^q(B)}\lesssim
\prod_{i=1}^{\ell}\big\|f_i\big\|_{L^{p_i}\log L(B)}\cdot\prod_{j=\ell+1}^m\big\|f_j\big\|_{L^{p_j}(B)}
\end{equation*}
holds. Here we suppose that $p_1=\cdots=p_{\ell}=s'$ and $p_{\ell+1}=\cdots=p_m>s'$.
\end{thm}

\begin{proof}[Proof of Theorem $\ref{thm7}$]
Taking into consideration \eqref{wanghua1}, we get
\begin{equation*}
\begin{split}
\big\|\mathcal{T}_{\Omega,\alpha;m}(\vec{f})\big\|_{L^q(B)}
&\lesssim\bigg(\int_{B}\big|\mathcal{M}_{s'}(\vec{f})(x)\big|^{p}dx\bigg)^{1/q}
\cdot\Big[\prod_{i=1}^m\|f_i\|_{L^{p_i}(B)}\Big]^{1-p/q}\\
&=\bigg(\int_{B}\big|\mathcal{M}(\vec{f})^{s'}(x)\big|^{p/{s'}}dx\bigg)^{1/q}
\cdot\Big[\prod_{i=1}^m\|f_i\|_{L^{p_i}(B)}\Big]^{1-p/q}.
\end{split}
\end{equation*}
Recalling the definition of the multilinear maximal operator $\mathcal{M}$, we have
\begin{equation}\label{wang830}
\mathcal{M}(f_1^{s'},f_2^{s'},\dots,f_m^{s'})(x)\leq \prod_{i=1}^m M(|f_i|^{s'})(x),\quad \forall~x\in\mathbb R^n.
\end{equation}
When $p_i>s'$ for some $1\leq i\leq m$, then by the strong-type boundedness of $M$, we get
\begin{equation}\label{wang831}
\big\|M(|f_i|^{s'})\big\|_{L^{p_i/{s'}}(B)}\leq C\big\||f_i|^{s'}\big\|_{L^{p_i/{s'}}(B)}=C\big\|f_i\big\|_{L^{p_i}(B)}^{s'}.
\end{equation}
On the other hand, when $p_i=s'$ for some $1\leq i\leq m$, it then follows from \eqref{Mpp} that
\begin{equation}\label{wang832}
\big\|M(|f_i|^{s'})\big\|_{L^{p_i/{s'}}(B)}=\big\|M(|f_i|^{p_i})\big\|_{L^1(B)}
\leq C\big\|f_i\big\|_{L^{p_i}\log L(B)}^{p_i}.
\end{equation}
Notice that
\begin{equation*}
\frac{1}{p/{s'}}=\frac{1}{p_1/{s'}}+\frac{1}{p_2/{s'}}+\cdots+\frac{1}{p_m/{s'}}.
\end{equation*}
When $\min\{p_1,p_2,\dots,p_m\}=s'$, without loss of generality, we may assume that
\begin{equation*}
p_1=p_2=\cdots=p_{\ell}=s',\quad \mbox{and} \quad p_{\ell+1}=p_{\ell+2}=\cdots=p_m>s'.
\end{equation*}
Applying H\"{o}lder's inequality for a product of $m$ functions and estimates \eqref{wang830},\eqref{wang831} and \eqref{wang832}, we can easily deduce that
\begin{equation*}
\begin{split}
\big\|\mathcal{T}_{\Omega,\alpha;m}(\vec{f})\big\|_{L^q(B)}
&\lesssim\Big[\big\|\mathcal{M}(f_1^{s'},f_2^{s'},\dots,f_m^{s'})\big\|_{L^{p/{s'}}(B)}\Big]^{(p/{s'})\cdot(1/q)}
\cdot\Big[\prod_{i=1}^m\|f_i\|_{L^{p_i}(B)}\Big]^{1-p/q}\\
&\leq\Big[\prod_{i=1}^m\big\|M(|f_i|^{s'})\big\|_{L^{p_i/{s'}}(B)}\Big]^{(p/{s'})\cdot(1/q)}
\cdot\Big[\prod_{i=1}^m\|f_i\|_{L^{p_i}(B)}\Big]^{1-p/q}\\
&\lesssim \Big[\prod_{i=1}^{\ell}\big\|f_i\big\|^{s'}_{L^{p_i}\log L(B)}
\cdot\prod_{j=\ell+1}^m\big\|f_j\big\|_{L^{p_j}(B)}^{s'}\Big]^{(p/{s'})\cdot(1/q)}\\
&\times\Big[\prod_{i=1}^m\|f_i\|_{L^{p_i}(B)}\Big]^{1-p/q}\\
&\leq \prod_{i=1}^{\ell}\big\|f_i\big\|_{L^{p_i}\log L(B)}\cdot\prod_{j=\ell+1}^m\big\|f_j\big\|_{L^{p_j}(B)},
\end{split}
\end{equation*}
where in the last inequality we have used \eqref{lplogl}. The proof of Theorem \ref{thm7} is now complete.
\end{proof}

\section{The critical case $p=n/{\alpha}$}
As in the proof of \eqref{pointwise}, we can also prove that $\mathcal{M}_{\Omega,\alpha;m}(\vec{f})$ is dominated
by $\mathcal{T}_{|\Omega|,\alpha;m}(\vec{f})$ for any $\vec{f}$. Indeed, for any $0<\alpha<mn$, $x\in\mathbb R^n$ and $r>0$, we have
\begin{equation*}
\begin{split}
\mathcal{T}_{|\Omega|,\alpha;m}(|\vec{f}|)(x)
&\geq
\int_{|x-y_1|<r}\cdots\int_{|x-y_m|<r}\frac{1}{\|(x-y_1,\dots,x-y_m)\|^{mn-\alpha}}
\prod_{i=1}^m\big|\Omega_i(x-y_i)f_i(y_i)\big|\,dy_i\\
&\geq\frac{1}{(mr)^{mn-\alpha}}\prod_{i=1}^m\int_{|x-y_i|<r}\big|\Omega_i(x-y_i)f_i(y_i)\big|\,dy_i\\
&=C_{m,n,\alpha}\prod_{i=1}^m\frac{1}{m(B(x,r))^{1-\alpha/{mn}}}\int_{|x-y_i|<r}\big|\Omega_i(x-y_i)f_i(y_i)\big|\,dy_i.
\end{split}
\end{equation*}
Thus, the desired result follows by taking the supremum for all $r>0$. As a direct consequence of Theorems \ref{thm1} and \ref{thm7}, we have the corresponding results for the multilinear fractional maximal operator $\mathcal{M}_{\Omega,\alpha;m}$.
\begin{thm}\label{maximalthm1}
Let $\vec{\Omega}=(\Omega_1,\Omega_2,\dots,\Omega_m)\in L^s(\mathbf{S}^{n-1})$ with $1<s\leq\infty$. Let $0<\alpha<mn$, $s'\leq p_1,p_2,\dots,p_m<\infty$,
\begin{equation*}
1/p=1/{p_1}+1/{p_2}+\cdots+1/{p_m}\quad \mathrm{and}\quad 1/q=1/p-\alpha/n.
\end{equation*}

$(i)$ If $s'<p_1,p_2,\dots,p_m<\infty$, then for every $p\in(s'/m,n/{\alpha})$, the following inequality
\begin{equation*}
\big\|\mathcal{M}_{\Omega,\alpha;m}(\vec{f})\big\|_{L^q}\lesssim\prod_{i=1}^m\big\|f_i\big\|_{L^{p_i}}
\end{equation*}
holds for all $\vec{f}=(f_1,f_2,\dots,f_m)\in L^{p_1}(\mathbb R^n)\times L^{p_2}(\mathbb R^n)\times \cdots\times L^{p_m}(\mathbb R^n)$.

$(ii)$ If $s'=\min\{p_1,p_2,\dots,p_m\}$, then for every $p\in[s'/m,n/{\alpha})$, the following inequality
\begin{equation*}
\big\|\mathcal{M}_{\Omega,\alpha;m}(\vec{f})\big\|_{L^{q,\infty}}\lesssim\prod_{i=1}^m\big\|f_i\big\|_{L^{p_i}}
\end{equation*}
holds for all $\vec{f}=(f_1,f_2,\dots,f_m)\in L^{p_1}(\mathbb R^n)\times L^{p_2}(\mathbb R^n)\times \cdots\times L^{p_m}(\mathbb R^n)$.

$(iii)$ If $s'=\min\{p_1,p_2,\dots,p_m\}$ and $p<n/{\alpha}$, then for any given ball $B\subset\mathbb R^n$, the following inequality
\begin{equation*}
\big\|\mathcal{M}_{\Omega,\alpha;m}(\vec{f})\big\|_{L^q(B)}\lesssim
\prod_{i=1}^{\ell}\big\|f_i\big\|_{L^{p_i}\log L(B)}\cdot\prod_{j=\ell+1}^m\big\|f_j\big\|_{L^{p_j}(B)}
\end{equation*}
holds. Here we suppose that $p_1=\cdots=p_{\ell}=s'$ and $p_{\ell+1}=\cdots=p_m>s'$.
\end{thm}

For the critical case $p=n/{\alpha}$, we can also show that the multilinear fractional maximal operator $\mathcal{M}_{\Omega,\alpha;m}$ is bounded from $L^{p_1}(\mathbb R^n)\times L^{p_2}(\mathbb R^n)\times\cdots\times L^{p_m}(\mathbb R^n)$ into $L^{\infty}(\mathbb R^n)$.

\begin{thm}\label{maximalthm2}
Let $\vec{\Omega}=(\Omega_1,\Omega_2,\dots,\Omega_m)\in L^s(\mathbf{S}^{n-1})$ with $1<s\leq\infty$. If
\begin{equation*}
0<\alpha<mn,\qquad  s'\leq p_1,p_2,\dots,p_m<\infty,
\end{equation*}
and
\begin{equation*}
{\alpha}/{n}=1/{p_1}+1/{p_2}+\cdots+1/{p_m},
\end{equation*}
then the following inequality
\begin{equation*}
\big\|\mathcal{M}_{\Omega,\alpha;m}(\vec{f})\big\|_{L^{\infty}}\lesssim\prod_{i=1}^m\big\|f_i\big\|_{L^{p_i}}
\end{equation*}
holds for all $\vec{f}=(f_1,f_2,\dots,f_m)\in L^{p_1}(\mathbb R^n)\times L^{p_2}(\mathbb R^n)\times\cdots\times L^{p_m}(\mathbb R^n)$.
\end{thm}

\begin{proof}[Proof of Theorem $\ref{maximalthm2}$]
Given a ball $B(x,r)$ with center $x\in\mathbb R^n$ and radius $r\in(0,\infty)$, by using H\"{o}lder's inequality and \eqref{omega88}, we get
\begin{equation*}
\begin{split}
&\prod_{i=1}^m\frac{1}{m(B(x,r))^{1-\alpha/{mn}}}\int_{|x-y_i|<r}|\Omega_i(x-y_i)|\cdot|f_i(y_i)|\,dy_i\\
&\leq\prod_{i=1}^m\frac{1}{m(B(x,r))^{1-\alpha/{mn}}}\bigg(\int_{|x-y_i|<r}|\Omega_i(x-y_i)|^sdy_i\bigg)^{1/s}
\bigg(\int_{|x-y_i|<r}|f_i(y_i)|^{s'}dy_i\bigg)^{1/{s'}}\\
&\lesssim\big\|\vec{\Omega}\big\|_{L^s(\mathbf{S}^{n-1})}
\prod_{i=1}^m\frac{1}{m(B(x,r))^{1-\alpha/{mn}}}
\cdot m(B(x,r))^{1/s}\bigg(\int_{|x-y_i|<r}|f_i(y_i)|^{s'}dy_i\bigg)^{1/{s'}}.
\end{split}
\end{equation*}

Two cases are considered below.

\textbf{Case 1:} $p_i=s'$ for some $1\leq i\leq m$. In this case, one has $p'_i=s$, then we certainly have
\begin{equation*}
m(B(x,r))^{1/s}\bigg(\int_{|x-y_i|<r}|f_i(y_i)|^{s'}dy_i\bigg)^{1/{s'}}
\leq m(B(x,r))^{1-1/{p_i}}\big\|f_i\big\|_{L^{p_i}}.
\end{equation*}

\textbf{Case 2:} $p_i>s'$ for some $1\leq i\leq m$. A direct computation shows that
\begin{equation*}
\begin{split}
\frac{1}{s'({p_i}/{s'})'}=\frac{1}{s'}\cdot\frac{p_i-s'}{p_i}=\frac{1}{s'}-\frac{1}{p_i}.
\end{split}
\end{equation*}
This, combined with H\"{o}lder's inequality, yields
\begin{equation*}
\begin{split}
\bigg(\int_{|x-y_i|<r}|f_i(y_i)|^{s'}dy_i\bigg)^{1/{s'}}
&\leq\bigg(\int_{B(x,r)}|f_i(y_i)|^{p_i}\,dy_i\bigg)^{1/{p_i}}\Big[m(B(x,r))\Big]^{\frac{1}{s'({p_i}/{s'})'}}\\
&=\bigg(\int_{B(x,r)}|f_i(y_i)|^{p_i}\,dy_i\bigg)^{1/{p_i}}\Big[m(B(x,r))\Big]^{1/{s'}-1/{p_i}},
\end{split}
\end{equation*}
and hence
\begin{equation*}
m(B(x,r))^{1/s}\bigg(\int_{|x-y_i|<r}|f_i(y_i)|^{s'}dy_i\bigg)^{1/{s'}}
\leq m(B(x,r))^{1-1/{p_i}}\big\|f_i\big\|_{L^{p_i}}.
\end{equation*}
Summing up the above estimates, we conclude that for any $x\in\mathbb R^n$ and $r>0$,
\begin{equation*}
\begin{split}
&\prod_{i=1}^m\frac{1}{m(B(x,r))^{1-\alpha/{mn}}}\int_{|x-y_i|<r}|\Omega_i(x-y_i)|\cdot|f_i(y_i)|\,dy_i\\
&\lesssim\|\vec{\Omega}\|_{L^s(\mathbf{S}^{n-1})}
\prod_{i=1}^m\frac{1}{m(B(x,r))^{1-\alpha/{mn}}}\times m(B(x,r))^{1-1/{p_i}}\big\|f_i\big\|_{L^{p_i}}\\
&=\|\vec{\Omega}\|_{L^s(\mathbf{S}^{n-1})}\prod_{i=1}^m\big\|f_i\big\|_{L^{p_i}}
\frac{1}{m(B(x,r))^{m-\alpha/n}}\cdot m(B(x,r))^{\sum_{i=1}^m(1-1/{p_i})}\\
&=\|\vec{\Omega}\|_{L^s(\mathbf{S}^{n-1})}\prod_{i=1}^m\big\|f_i\big\|_{L^{p_i}},
\end{split}
\end{equation*}
where in the last step we have used the fact that $\sum_{i=1}^m 1/{p_i}=\alpha/n$. By taking the supremum for all $r>0$, we are done.
\end{proof}

A locally integrable function $f$ is said to be in the bounded mean oscillation space $\mathrm{BMO}(\mathbb R^n)$ (see \cite{john}), if
\begin{equation*}
\|f\|_{\mathrm{BMO}}:=\sup_{B\subset\mathbb R^n}\frac{1}{m(B)}\int_B|f(x)-f_B|\,dx<+\infty,
\end{equation*}
where $f_B$ denotes the average of $f$ on $B$, i.e.,
\begin{equation*}
f_B:=\frac{1}{m(B)}\int_B f(y)\,dy
\end{equation*}
and the supremum is taken over all balls $B$ in $\mathbb R^n$. Modulo constants, the space $\mathrm{BMO}(\mathbb R^n)$ is a Banach space with respect to the norm $\|\cdot\|_{\mathrm{BMO}}$. For the critical case $p=n/{\alpha}$, it can be shown that the multilinear fractional integral operator $\mathcal{T}_{\Omega,\alpha;m}$ is bounded from $L^{p_1}(\mathbb R^n)\times L^{p_2}(\mathbb R^n)\times\cdots\times L^{p_m}(\mathbb R^n)$ into $\mathrm{BMO}(\mathbb R^n)$, provided $\Omega_1=\cdots=\Omega_m\equiv1$.

\begin{thm}\label{maximalthm3}
Let $0<\alpha<mn$, $1\leq p_1,p_2,\dots,p_m<\infty$, and
\begin{equation*}
{\alpha}/{n}=1/{p_1}+1/{p_2}+\cdots+1/{p_m}.
\end{equation*}
Suppose that $\Omega_1=\Omega_2=\cdots=\Omega_m\equiv1$, then the following inequality
\begin{equation*}
\big\|\mathcal{T}_{\Omega,\alpha;m}(\vec{f})\big\|_{\mathrm{BMO}}\lesssim\prod_{i=1}^m\big\|f_i\big\|_{L^{p_i}}
\end{equation*}
holds for all $\vec{f}=(f_1,f_2,\dots,f_m)\in L^{p_1}(\mathbb R^n)\times L^{p_2}(\mathbb R^n)\times\cdots\times L^{p_m}(\mathbb R^n)$.
\end{thm}

This result was first given in \cite{tang}(as far as we know). For the linear (weighted) case, see \cite{muckenhoupt2}.

\section{Boundedness of $\mathcal{T}_{\Omega,\alpha;m}$ and $\mathcal{M}_{\Omega,\alpha;m}$ on Morrey spaces}
\label{3}
In this section, we investigate the boundedness properties of multilinear fractional integral operator and fractional maximal operator on Morrey spaces.
Let us begin with a lemma, which can be found in \cite{chiarenza} and \cite{mizuhara}. For the weighted version, see \cite{komori}.
\begin{lem}\label{MM}
Let $1<p<\infty$ and $0<\kappa<1$. Then the Hardy--Littlewood maximal operator $M$ is bounded on $L^{p,\kappa}(\mathbb R^n)$, and bounded from $L^{1,\kappa}(\mathbb R^n)$ into $WL^{1,\kappa}(\mathbb R^n)$.
\end{lem}

Based on Lemma \ref{MM}, we can establish the corresponding estimates for the multilinear maximal operator $\mathcal{M}$ on Morrey spaces, which plays an important role in our proofs of main theorems.
\begin{lem}\label{multimaxmorrey}
Let $0<\kappa<1$, $1\leq p_1,p_2,\dots,p_m<\infty$, $0<p<\infty$ and
\begin{equation*}
1/p=1/{p_1}+1/{p_2}+\cdots+1/{p_m}.
\end{equation*}
\begin{enumerate}
  \item If $1<p_i<\infty$, $i=1,2,\dots,m$, then the inequality
  \begin{equation*}
    \big\|\mathcal{M}(\vec{f})\big\|_{L^{p,\kappa}}\leq C\prod_{i=1}^m\big\|f_i\big\|_{L^{p_i,\kappa}}
  \end{equation*}
  holds for every $\vec{f}=(f_1,f_2,\dots,f_m)\in L^{p_1,\kappa}(\mathbb R^n)\times L^{p_2,\kappa}(\mathbb R^n)\times \cdots\times L^{p_m,\kappa}(\mathbb R^n)$.
  \item If $\min\{p_1,p_2,\dots,p_m\}=1$, then the inequality
\begin{equation*}
\big\|\mathcal{M}(\vec{f})\big\|_{WL^{p,\kappa}}\leq C\prod_{i=1}^m\big\|f_i\big\|_{L^{p_i,\kappa}}
\end{equation*}
holds for every $\vec{f}=(f_1,f_2,\dots,f_m)\in L^{p_1,\kappa}(\mathbb R^n)\times L^{p_2,\kappa}(\mathbb R^n)\times \cdots\times L^{p_m,\kappa}(\mathbb R^n)$.
\end{enumerate}
\end{lem}

\begin{proof}
We first claim that the following inequality
\begin{equation}\label{multimowang}
\Big\|\prod_{i=1}^m \mathcal{F}_i\Big\|_{L^{p,\kappa}}\leq \prod_{i=1}^m\big\|\mathcal{F}_i\big\|_{L^{p_i,\kappa}}
\end{equation}
holds for any $\mathcal{F}_i\in L^{p_i,\kappa}$ with $1\leq p_i<\infty$ and $0<\kappa<1$, $i=1,2,\dots,m$. In fact, for any given ball $\mathcal{B}$ in $\mathbb R^n$, by using H\"{o}lder's inequality and the definition of $L^{p,\kappa}(\mathbb R^n)$, we can deduce that
\begin{equation*}
\begin{split}
&\bigg(\frac{1}{m(\mathcal{B})^{\kappa}}\int_{\mathcal{B}}\big|\mathcal{F}_1(x)\mathcal{F}_2(x)\cdots \mathcal{F}_m(x)\big|^p\,dx\bigg)^{1/p}\\
&=\frac{1}{m(\mathcal{B})^{\kappa/p}}\Big\|\prod_{i=1}^m \mathcal{F}_i\Big\|_{L^p(\mathcal{B})}
\leq\frac{1}{m(\mathcal{B})^{\kappa/p}}\prod_{i=1}^m\big\|\mathcal{F}_i\big\|_{L^{p_i}(\mathcal{B})}\\
&\leq\frac{1}{m(\mathcal{B})^{\kappa/p}}\prod_{i=1}^m\big\|\mathcal{F}_i\big\|_{L^{p_i,\kappa}}
\cdot m(\mathcal{B})^{\kappa/{p_i}}=\prod_{i=1}^m\big\|\mathcal{F}_i\big\|_{L^{p_i,\kappa}},
\end{split}
\end{equation*}
where in the last step we have used the fact that $1/p=1/{p_1}+1/{p_2}+\cdots+1/{p_m}$. Thus, the desired inequality \eqref{multimowang} follows by taking the supremum over all balls $\mathcal{B}$ in $\mathbb R^n$. Recall that by definition,
\begin{equation}\label{wang830729}
\mathcal{M}(f_1,f_2,\dots,f_m)(x)\leq \prod_{i=1}^m M(f_i)(x),\quad \forall~x\in\mathbb R^n.
\end{equation}
Therefore, it follows directly from Lemma \ref{MM} and \eqref{multimowang} that
\begin{equation*}
\begin{split}
\big\|\mathcal{M}(\vec{f})\big\|_{L^{p,\kappa}}&\leq \Big\|\prod_{i=1}^m M(f_i)\Big\|_{L^{p,\kappa}}\\
&\leq \prod_{i=1}^m\big\|M(f_i)\big\|_{L^{p_i,\kappa}}\leq C\prod_{i=1}^m\big\|f_i\big\|_{L^{p_i,\kappa}}.
\end{split}
\end{equation*}
Let us now turn to the proof of part (2). Without loss of generality, we may assume that $p_1=\cdots=p_{\ell}=1$ and $p_{\ell+1}=\cdots=p_m>1$, $1\leq\ell\leq m$. Applying H\"{o}lder's inequality for weak $L^p$ spaces (see \cite[p.15]{grafakos}) and the definition of $WL^{p,\kappa}(\mathbb R^n)$, we can deduce that
\begin{equation*}
\begin{split}
&\frac{1}{m(\mathcal{B})^{\kappa/p}}\Big\|\prod_{i=1}^m M(f_i)\cdot\chi_{\mathcal{B}}\Big\|_{L^{p,\infty}}\\
&\lesssim\frac{1}{m(\mathcal{B})^{\kappa/p}}\prod_{i=1}^{\ell}\big\|M(f_i)\cdot\chi_{\mathcal{B}}\big\|_{L^{p_i,\infty}}
\cdot\prod_{j=\ell+1}^m\big\|M(f_j)\cdot\chi_{\mathcal{B}}\big\|_{L^{p_j}}\\
&\leq\frac{1}{m(\mathcal{B})^{\kappa/p}}\prod_{i=1}^{\ell}\big\|M(f_i)\big\|_{WL^{p_i,\kappa}}m(\mathcal{B})^{\kappa/{p_i}}
\cdot\prod_{j=\ell+1}^m\big\|M(f_j)\big\|_{L^{p_j,\kappa}}m(\mathcal{B})^{\kappa/{p_j}}\\
&=\prod_{i=1}^{\ell}\big\|M(f_i)\big\|_{WL^{p_i,\kappa}}\cdot
\prod_{j=\ell+1}^m\big\|M(f_j)\big\|_{L^{p_j,\kappa}},
\end{split}
\end{equation*}
where the last equality follows from the fact that $1/{p_1}+\cdots+1/{p_{\ell}}+1/{p_{\ell+1}}+\cdots+1/{p_m}=1/p$. Therefore, by the pointwise estimate \eqref{wang830729} and Lemma \ref{MM}, we get
\begin{equation*}
\begin{split}
&\frac{1}{m(\mathcal{B})^{\kappa/p}}\big\|\mathcal{M}(\vec{f})\cdot\chi_{\mathcal{B}}\big\|_{L^{p,\infty}}
\leq\frac{1}{m(\mathcal{B})^{\kappa/p}}\Big\|\prod_{i=1}^m M(f_i)\cdot\chi_{\mathcal{B}}\Big\|_{L^{p,\infty}}\\
\leq& C\prod_{i=1}^{\ell}\big\|M(f_i)\big\|_{WL^{p_i,\kappa}}\cdot
\prod_{j=\ell+1}^m\big\|M(f_j)\big\|_{L^{p_j,\kappa}}\\
\leq& C\prod_{i=1}^{\ell}\big\|f_i\big\|_{L^{p_i,\kappa}}\cdot\prod_{j=\ell+1}^m\big\|f_j\big\|_{L^{p_j,\kappa}}
=C\prod_{i=1}^{m}\big\|f_i\big\|_{L^{p_i,\kappa}}.
\end{split}
\end{equation*}
Finally, by taking the supremum over all balls $\mathcal{B}$ in $\mathbb R^n$, we are done.
\end{proof}

\begin{thm}\label{thm2}
Let $\vec{\Omega}=(\Omega_1,\Omega_2,\dots,\Omega_m)\in L^s(\mathbf{S}^{n-1})$ with $1<s\leq\infty$. Let $0<\alpha<mn$, $s'\leq p_1,p_2,\dots,p_m<\infty$, $0<\kappa<1-{(\alpha p)}/n$,
\begin{equation*}
1/p=1/{p_1}+1/{p_2}+\cdots+1/{p_m}\quad \mathrm{and}\quad 1/q=1/p-\alpha/{[n(1-\kappa)]}.
\end{equation*}

$(i)$ If $s'<p_1,p_2,\dots,p_m<\infty$, then for every $p\in(s'/m,n/{\alpha})$, the following inequality
\begin{equation*}
\big\|\mathcal{T}_{\Omega,\alpha;m}(\vec{f})\big\|_{L^{q,\kappa}}\lesssim\prod_{i=1}^m\big\|f_i\big\|_{L^{p_i,\kappa}}
\end{equation*}
holds for all $\vec{f}=(f_1,f_2,\dots,f_m)\in L^{p_1,\kappa}(\mathbb R^n)\times L^{p_2,\kappa}(\mathbb R^n)\times \cdots\times L^{p_m,\kappa}(\mathbb R^n)$.

$(ii)$ If $s'=\min\{p_1,p_2,\dots,p_m\}$, then for every $p\in[s'/m,n/{\alpha})$, the following inequality
\begin{equation*}
\big\|\mathcal{T}_{\Omega,\alpha;m}(\vec{f})\big\|_{WL^{q,\kappa}}\lesssim\prod_{i=1}^m\big\|f_i\big\|_{L^{p_i,\kappa}}
\end{equation*}
holds for all $\vec{f}=(f_1,f_2,\dots,f_m)\in L^{p_1,\kappa}(\mathbb R^n)\times L^{p_2,\kappa}(\mathbb R^n)\times \cdots\times L^{p_m,\kappa}(\mathbb R^n)$.
\end{thm}

\begin{rem}
For the linear case $m=1$, a few historical remarks are illustrated below.
\begin{itemize}
  \item In \cite{peetre}, Peetre established the boundedness properties of the Riesz potential operator $I_{\alpha}=\mathcal{I}_{\alpha;1}$ and fractional maximal operator $M_{\alpha}=\mathcal{M}_{\alpha;1}$ on Morrey spaces. Let $0<\alpha<n$, $1<p<n/{\alpha}$ and $1/q=1/p-\alpha/n$. Peetre showed that if $0<\kappa<p/q$, then the fractional maximal operator $M_\alpha$ and the Riesz potential operator $I_{\alpha}$ are bounded from $L^{p,\kappa}(\mathbb R^n)$ to $L^{q,\kappa^{*}}(\mathbb R^n)$ with $\kappa^{*}:={(\kappa q)}/p$. Moreover, if $p=1$ and $0<\kappa<1/q$, then the operators $M_{\alpha}$ and $I_{\alpha}$ are bounded from $L^{1,\kappa}(\mathbb R^n)$ to $WL^{q,\kappa^{*}}(\mathbb R^n)$ with $q:=n/{(n-\alpha)}$ and $\kappa^{*}:=\kappa q$ (see also \cite{komori} for the weighted version).
  \item In \cite{adams}, Adams showed that if $0<\kappa<1-{(\alpha p)}/n$ and $1/q=1/p-\alpha/{[n(1-\kappa)]}$, then the fractional maximal operator $M_\alpha$ and the Riesz potential operator $I_{\alpha}$ are bounded from $L^{p,\kappa}(\mathbb R^n)$ to $L^{q,\kappa}(\mathbb R^n)$, when $0<\alpha<n$ and $1<p<n/{\alpha}$. Moreover, if $p=1$ and $0<\kappa<1-\alpha/n$, then the operators $M_{\alpha}$ and $I_{\alpha}$ are bounded from $L^{1,\kappa}(\mathbb R^n)$ to $WL^{q,\kappa}(\mathbb R^n)$ with $1/q:=1-\alpha/{[n(1-\kappa)]}$.
      It is worth pointing out that Adams's result is an essential improvement of Peetre's result. It can be shown that for $(q,\kappa)\in(1,\infty)\times(0,1)$, the inclusion relation
\begin{equation}\label{inclusion1}
L^{q,\kappa}(\mathbb R^n)\subset L^{q^{*},\kappa^{*}}(\mathbb R^n)
\end{equation}
holds for $q^{*}<q$ and $1-\kappa^{*}=(1-\kappa)(q^{*}/q)$, and
\begin{equation*}
\big\|f\big\|_{L^{q^{*},\kappa^{*}}}\leq\big\|f\big\|_{L^{q,\kappa}}.
\end{equation*}
Indeed, by H\"{o}lder's inequality, we know that for any ball $B$ in $\mathbb R^n$ and $1\leq q^{*}<q<\infty$,
\begin{equation*}
\begin{split}
&\bigg(\frac{1}{m(B)^{\kappa^{*}}}\int_B|f(x)|^{q^*}\,dx\bigg)^{1/{q^*}}\\
&=m(B)^{{(1-\kappa^*)}/{q^*}}\bigg(\frac{1}{m(B)}\int_B|f(x)|^{q^*}\,dx\bigg)^{1/{q^*}}\\
&\leq m(B)^{{(1-\kappa^*)}/{q^*}}\bigg(\frac{1}{m(B)}\int_B|f(x)|^{q}\,dx\bigg)^{1/{q}}\\
&=\bigg(\frac{1}{m(B)^{\kappa}}\int_B|f(x)|^{q}\,dx\bigg)^{1/{q}},
\end{split}
\end{equation*}
where in the last equality we have used the fact that $\kappa/q=1/q-{(1-\kappa^*)}/{q^*}$. This gives \eqref{inclusion1}. Moreover, using H\"{o}lder's inequality for weak $L^q$ spaces and repeating the argument above, we can also obtain
\begin{equation}\label{inclusion2}
WL^{q,\kappa}(\mathbb R^n)\subset WL^{q^{*},\kappa^{*}}(\mathbb R^n)
\end{equation}
whenever $1\leq q^{*}<q<\infty$ and $1-\kappa^{*}=(1-\kappa)(q^{*}/q)$, and
\begin{equation*}
\big\|f\big\|_{WL^{q^{*},\kappa^{*}}}\leq\big\|f\big\|_{WL^{q,\kappa}}.
\end{equation*}
Thus, by \eqref{inclusion1} and \eqref{inclusion2}, we can see that for given $f\in L^{p,\kappa}(\mathbb R^n)$,
\begin{equation*}
I_{\alpha}(f)(\mathrm{or}~ M_{\alpha}(f))\in L^{q,\kappa}(\mathbb R^n)\Longrightarrow
I_{\alpha}(f)(\mathrm{or}~ M_{\alpha}(f))\in L^{q^{*},\kappa^{*}}(\mathbb R^n)
\end{equation*}
whenever $1/{q^*}=1/p-\alpha/n$ and $\kappa^{*}={(\kappa q^{*})}/p$, and for $f\in L^{1,\kappa}(\mathbb R^n)$,
\begin{equation*}
I_{\alpha}(f)(\mathrm{or}~ M_{\alpha}(f))\in WL^{q,\kappa}(\mathbb R^n)\Longrightarrow
I_{\alpha}(f)(\mathrm{or}~ M_{\alpha}(f))\in WL^{q^{*},\kappa^{*}}(\mathbb R^n)
\end{equation*}
whenever $q^{*}=n/{(n-\alpha)}$ and $\kappa^{*}=\kappa {q^*}$.
 \item Moreover, in \cite{lu2}, Lu, Yang and Zhou showed that if $0<\kappa<p/q$, $s'<p<n/{\alpha}$ and $1/q=1/p-{\alpha}/n$, then the fractional integral operator $T_{\Omega,\alpha}=\mathcal{T}_{\Omega,\alpha;1}$ is bounded from $L^{p,\kappa}(\mathbb R^n)$ into $L^{q,\kappa^{*}}(\mathbb R^n)$ with $\kappa^{*}={(\kappa q)}/p$, when $0<\alpha<n$ and $\Omega\in L^s(\mathbf{S}^{n-1})$ with $1<s\leq\infty$ (see also \cite{mizuhara} and \cite{wang} for the weighted case).
\end{itemize}
\end{rem}

\begin{proof}[Proof of Theorem $\ref{thm2}$]
The idea of the proof is due to Adams (see \cite{adams} and \cite{adams1}). We will prove the following result.
\begin{equation}\label{wanghua2}
\begin{split}
\big|\mathcal{T}_{\Omega,\alpha;m}(\vec{f})(x)\big|\lesssim
\Big[\mathcal{M}_{s'}(\vec{f})(x)\Big]^{p/q}\cdot\Big[\prod_{i=1}^m\big\|f_i\big\|_{L^{p_i,\kappa}}\Big]^{1-p/q},
\quad \forall~x\in\mathbb R^n.
\end{split}
\end{equation}
Let $\vec{f}\in L^{p_1,\kappa}(\mathbb R^n)\times \cdots\times L^{p_m,\kappa}(\mathbb R^n)$ with $s'/m\leq p<n/{\alpha}$ and $0<\kappa<1-{(\alpha p)}/n$. In order to prove \eqref{wanghua2}, for any $x\in\mathbb R^n$, we write
\begin{equation*}
\begin{split}
&\big|\mathcal{T}_{\Omega,\alpha;m}(\vec{f})(x)\big|\leq
\int_{(\mathbb R^n)^m}\frac{|\Omega_1(x-y_1)\cdots
\Omega_m(x-y_m)|}{\|(x-y_1,\dots,x-y_m)\|^{mn-\alpha}}
\big|f_1(y_1)\cdots f_m(y_m)\big|\,dy_1\cdots dy_m\\
&=\int_{\|(x-y_1,\dots,x-y_m)\|<\sigma}\frac{|\Omega_1(x-y_1)\cdots\Omega_m(x-y_m)|}{\|(x-y_1,\dots,x-y_m)\|^{mn-\alpha}}
\big|f_1(y_1)\cdots f_m(y_m)\big|\,dy_1\cdots dy_m\\
&+\int_{\|(x-y_1,\dots,x-y_m)\|\geq\sigma}\frac{|\Omega_1(x-y_1)\cdots\Omega_m(x-y_m)|}{\|(x-y_1,\dots,x-y_m)\|^{mn-\alpha}}
\big|f_1(y_1)\cdots f_m(y_m)\big|\,dy_1\cdots dy_m\\
&:=\mathrm{III+IV},
\end{split}
\end{equation*}
where $\sigma$ is a positive number to be fixed below. By using the same argument involving $\mathrm{I}$ as in Theorem \ref{thm1}, we then have
\begin{equation*}
\begin{split}
\mathrm{III}&=\sum_{j=1}^\infty\int_{2^{-j}\sigma\leq\|(x-y_1,\dots,x-y_m)\|<2^{-j+1}\sigma}
\frac{|\Omega_1(x-y_1)\cdots\Omega_m(x-y_m)|}{\|(x-y_1,\dots,x-y_m)\|^{mn-\alpha}}\prod_{i=1}^m\big|f_i(y_i)\big|\,dy_i\\
&\lesssim\sum_{j=1}^\infty\frac{1}{(2^{-j}\sigma)^{mn-\alpha}}\cdot\prod_{i=1}^m
\big(2^{-j+1}\sigma\big)^{n}\mathcal{M}_{s'}(\vec{f})(x)\lesssim\sigma^{\alpha}\mathcal{M}_{s'}(\vec{f})(x).
\end{split}
\end{equation*}
Let us now give the estimate for the other term $\mathrm{IV}$. It is easy to see that
\begin{equation*}
\begin{split}
\mathrm{IV}&=\sum_{j=1}^\infty\int_{2^{j-1}\sigma\leq\|(x-y_1,\dots,x-y_m)\|<2^j\sigma}
\frac{|\Omega_1(x-y_1)\cdots\Omega_m(x-y_m)|}{\|(x-y_1,\dots,x-y_m)\|^{mn-\alpha}}\prod_{i=1}^m\big|f_i(y_i)\big|\,dy_i\\
&\leq\sum_{j=1}^\infty\frac{1}{(2^{j-1}\sigma)^{mn-\alpha}}
\int_{2^{j-1}\sigma\leq\|(x-y_1,\dots,x-y_m)\|<2^j\sigma}\prod_{i=1}^m|\Omega_i(x-y_i)|\cdot|f_i(y_i)|\,dy_i\\
&\leq\sum_{j=1}^\infty\frac{1}{(2^{j-1}\sigma)^{mn-\alpha}}
\prod_{i=1}^m\int_{|x-y_i|<2^j\sigma}|\Omega_i(x-y_i)|\cdot|f_i(y_i)|\,dy_i,
\end{split}
\end{equation*}
where in the last inequality we have used \eqref{geq}. Since $\Omega_i\in L^s(\mathbf{S}^{n-1})$ with $s\geq p'_i$, then $\Omega_i\in L^{p'_i}(\mathbf{S}^{n-1})$. Using H\"{o}lder's inequality and the estimate \eqref{omega88}, the last integral can be estimated as follows:
\begin{equation}\label{maines4}
\begin{split}
&\int_{|x-y_i|<2^j\sigma}|\Omega_i(x-y_i)|\cdot|f_i(y_i)|\,dy_i\\
&\leq\bigg(\int_{|x-y_i|<2^j\sigma}|\Omega_i(x-y_i)|^{p'_i}\,dy_i\bigg)^{1/{p'_i}}
\bigg(\int_{|x-y_i|<2^j\sigma}|f_i(y_i)|^{p_i}\,dy_i\bigg)^{1/{p_i}}\\
&\lesssim\|\Omega_i\|_{L^s(\mathbf{S}^{n-1})}\big(2^j\sigma\big)^{n/{p'_i}}
\cdot\big(2^j\sigma\big)^{{\kappa n}/{p_i}}\big\|f_i\big\|_{L^{p_i,\kappa}},
\end{split}
\end{equation}
for each integer $j\geq1$. Therefore, by \eqref{maines4}, \eqref{prod1} and \eqref{equality1}, we get
\begin{equation*}
\begin{split}
\mathrm{IV}&\lesssim\big\|\vec{\Omega}\big\|_{L^s(\mathbf{S}^{n-1})}
\sum_{j=1}^\infty\frac{1}{(2^{j-1}\sigma)^{mn-\alpha}}\prod_{i=1}^m\big(2^j\sigma\big)^{n/{p'_i}}
\cdot\big(2^j\sigma\big)^{{\kappa n}/{p_i}}\big\|f_i\big\|_{L^{p_i,\kappa}}\\
&=\big\|\vec{\Omega}\big\|_{L^s(\mathbf{S}^{n-1})}\sum_{j=1}^\infty\frac{1}{(2^{j-1}\sigma)^{mn-\alpha}}
\cdot\big(2^j\sigma\big)^{\sum_{i=1}^m n/{p'_i}}\cdot\big(2^j\sigma\big)^{\sum_{i=1}^m{\kappa n}/{p_i}}
\prod_{i=1}^m\big\|f_i\big\|_{L^{p_i,\kappa}}\\
\end{split}
\end{equation*}
\begin{equation*}
\begin{split}
&\lesssim\big\|\vec{\Omega}\big\|_{L^s(\mathbf{S}^{n-1})}
\sum_{j=1}^\infty\frac{1}{(2^{j}\sigma)^{n/p-{\kappa n}/p-\alpha}}\prod_{i=1}^m\big\|f_i\big\|_{L^{p_i,\kappa}}\\
&\lesssim\sigma^{\alpha-{(1-\kappa)n}/p}\prod_{i=1}^m\big\|f_i\big\|_{L^{p_i,\kappa}},
\end{split}
\end{equation*}
where the last inequality follows from the fact that ${(1-\kappa)}/p-\alpha/n>0$. Collecting all these estimates for $\mathrm{III}$ and $\mathrm{IV}$, we conclude that for any $x\in\mathbb R^n$,
\begin{equation}\label{main2}
\big|\mathcal{T}_{\Omega,\alpha;m}(\vec{f})(x)\big|
\lesssim\Big[\sigma^{\alpha}\mathcal{M}_{s'}(\vec{f})(x)+\sigma^{\alpha-{(1-\kappa)n}/p}
\prod_{i=1}^m\big\|f_i\big\|_{L^{p_i,\kappa}}\Big].
\end{equation}
We now choose $\sigma>0$ such that
\begin{equation*}
\sigma^{\alpha}\mathcal{M}_{s'}(\vec{f})(x)=\sigma^{\alpha-{(1-\kappa)n}/p}
\prod_{i=1}^m\big\|f_i\big\|_{L^{p_i,\kappa}},
\end{equation*}
that is,
\begin{equation*}
\sigma^{{(1-\kappa)n}/p}=\frac{\prod_{i=1}^m\|f_i\|_{L^{p_i,\kappa}}}{\mathcal{M}_{s'}(\vec{f})(x)}.
\end{equation*}
Putting this value of $\sigma$ back into \eqref{main2} and noting that
\begin{equation*}
1-{\alpha p}/{(1-\kappa)n}=p\cdot\big[1/p-\alpha/{(1-\kappa)n}\big]=p/q,
\end{equation*}
we obtain
\begin{equation*}
\begin{split}
\big|\mathcal{T}_{\Omega,\alpha;m}(\vec{f})(x)\big|\lesssim\sigma^{\alpha}\mathcal{M}_{s'}(\vec{f})(x)
&=\Big[\frac{\prod_{i=1}^m\|f_i\|_{L^{p_i,\kappa}}}{\mathcal{M}_{s'}(\vec{f})(x)}\Big]^{{\alpha p}/{(1-\kappa)n}}
\cdot\mathcal{M}_{s'}(\vec{f})(x)\\
&=\Big[\mathcal{M}_{s'}(\vec{f})(x)\Big]^{p/q}\cdot\Big[\prod_{i=1}^m\big\|f_i\big\|_{L^{p_i,\kappa}}\Big]^{1-p/q}.
\end{split}
\end{equation*}
Thus, \eqref{wanghua2} holds. The conclusion of Theorem \ref{thm2} then follows from \eqref{wanghua2} and the boundedness of $\mathcal{M}$ on Morrey spaces (Lemma \ref{multimaxmorrey}).
\begin{itemize}
  \item If $s'<p_1,p_2,\dots,p_m$ and $p<n/{\alpha}$, we first observe that
\begin{equation}\label{fcase}
\begin{split}
\big\||f_i|^{s'}\big\|_{L^{{p_i}/{s'},\kappa}}
=&\sup_{\mathcal{B}\subseteq\mathbb R^n}\bigg(\frac{1}{m(\mathcal{B})^{\kappa}}
\int_{\mathcal{B}}|f_i(x)|^{s'\cdot({p_i}/{s'})}\,dx\bigg)^{{s'}/{p_i}}\\
=&\bigg[\sup_{\mathcal{B}\subseteq\mathbb R^n}
\bigg(\frac{1}{m(\mathcal{B})^{\kappa}}\int_{\mathcal{B}}|f_i(x)|^{p_i}\,dx\bigg)^{1/{p_i}}\bigg]^{s'}\\
=&\Big[\big\|f_i\big\|_{L^{p_i,\kappa}}\Big]^{s'}.
\end{split}
\end{equation}
An application of \eqref{wanghua2} yields
\begin{equation*}
\begin{split}
\big\|\mathcal{T}_{\Omega,\alpha;m}(\vec{f})\big\|_{L^{q,\kappa}}
&\lesssim\sup_{\mathcal{B}\subseteq\mathbb R^n}\bigg(\frac{1}{m(\mathcal{B})^{\kappa}}
\int_{\mathcal{B}}\big|\mathcal{M}_{s'}(\vec{f})(x)\big|^{p}dx\bigg)^{1/q}
\cdot\Big[\prod_{i=1}^m\big\|f_i\big\|_{L^{p_i,\kappa}}\Big]^{1-p/q}\\
&=\sup_{\mathcal{B}\subseteq\mathbb R^n}\bigg(\frac{1}{m(\mathcal{B})^{\kappa}}\int_{\mathcal{B}}
\big|\mathcal{M}(f_1^{s'},f_2^{s'},\dots,f_m^{s'})(x)\big|^{p/{s'}}\,dx\bigg)^{1/{q}}
\cdot\Big[\prod_{i=1}^m\big\|f_i\big\|_{L^{p_i,\kappa}}\Big]^{1-p/q}.
\end{split}
\end{equation*}
Moreover, in view of \eqref{fcase}, \eqref{prod1} and part (1) of Lemma \ref{multimaxmorrey}, we have
\begin{equation*}
\begin{split}
\big\|\mathcal{T}_{\Omega,\alpha;m}(\vec{f})\big\|_{L^{q,\kappa}}
&\lesssim\Big[\prod_{i=1}^m\big\||f_i|^{s'}\big\|_{L^{{p_i}/{s'},\kappa}}\Big]^{p/{(s'q)}}
\cdot\Big[\prod_{i=1}^m\big\|f_i\big\|_{L^{p_i,\kappa}}\Big]^{1-p/q}\\
&=\Big[\prod_{i=1}^m\big\|f_i\big\|_{L^{p_i,\kappa}}^{s'}\Big]^{p/{(s'q)}}
\cdot\Big[\prod_{i=1}^m\big\|f_i\big\|_{L^{p_i,\kappa}}\Big]^{1-p/q}\\
&=\Big[\prod_{i=1}^m\big\|f_i\big\|_{L^{p_i,\kappa}}\Big]^{p/q}
\cdot\Big[\prod_{i=1}^m\big\|f_i\big\|_{L^{p_i,\kappa}}\Big]^{1-p/q}
=\prod_{i=1}^m\big\|f_i\big\|_{L^{p_i,\kappa}}.
\end{split}
\end{equation*}
This shows part $(i)$.
  \item If $s'=\min\{p_1,p_2,\dots,p_m\}$ and $p<n/{\alpha}$, we then follow the same arguments as above and find that \eqref{fcase} remains true in this case. By using the estimate \eqref{wanghua2}, for any given ball $\mathcal{B}$ in $\mathbb R^n$, then we have
\begin{equation*}
\begin{split}
&\frac{1}{m(\mathcal{B})^{\kappa/q}}
\sup_{\lambda>0}\lambda\cdot m\big(\big\{x\in \mathcal{B}:\big|\mathcal{T}_{\Omega,\alpha;m}(\vec{f})(x)\big|>\lambda\big\}\big)^{1/q}\\
&\leq\frac{1}{m(\mathcal{B})^{\kappa/q}}\sup_{\lambda>0}\lambda\cdot m\bigg(\bigg\{x\in \mathcal{B}:\big|\mathcal{M}_{s'}(\vec{f})(x)\big|^{p/q}>\frac{\lambda}{C\prod_{i=1}^m\|f_i\|^{1-p/q}_{L^{p_i,\kappa}}}\bigg\}\bigg)^{1/q}\\
&=\bigg[\frac{1}{m(\mathcal{B})^{(\kappa s')/p}}\sup_{\lambda>0}\lambda^{{(s'q)}/p}
\cdot m\bigg(\bigg\{x\in \mathcal{B}:\big|\mathcal{M}(\vec{f})^{s'}(x)\big|>
\bigg(\frac{\lambda}{C\prod_{i=1}^m\|f_i\|^{1-p/q}_{L^{p_i,\kappa}}}\bigg)^{{(s'q)}/p}\bigg\}\bigg)^{\frac{1}{p/{s'}}}\bigg]^{\frac{p}{(s'q)}}.
\end{split}
\end{equation*}
In view of \eqref{fcase}, \eqref{prod1} and part (2) of Lemma \ref{multimaxmorrey}, the above expression is bounded by
\begin{equation*}
\begin{split}
&C\bigg[\Big(\prod_{i=1}^m\big\|f_i\big\|^{1-p/q}_{L^{p_i,\kappa}}\Big)^{{(s'q)}/p}\bigg]^{\frac{p}{(s'q)}}
\cdot\bigg[\Big\|\mathcal{M}(f_1^{s'},f_2^{s'},\dots,f_m^{s'})\Big\|_{WL^{p/{s'},\kappa}}\bigg]^{\frac{p}{(s'q)}}\\
&\leq C\bigg[\prod_{i=1}^m\big\|f_i\big\|_{L^{p_i,\kappa}}\bigg]^{1-p/q}
\cdot\bigg[\prod_{i=1}^m\big\||f_i|^{s'}\big\|_{L^{{p_i}/{s'},\kappa}}\bigg]^{\frac{p}{(s'q)}}\\
&=C\bigg[\prod_{i=1}^m\big\|f_i\big\|_{L^{p_i,\kappa}}\bigg]^{1-p/q}
\cdot\bigg[\prod_{i=1}^m\big\|f_i\big\|_{L^{p_i,\kappa}}^{s'}\bigg]^{\frac{p}{(s'q)}}\\
&=C\bigg[\prod_{i=1}^m\big\|f_i\big\|_{L^{p_i,\kappa}}\bigg]^{1-p/q}
\cdot\bigg[\prod_{i=1}^m\big\|f_i\big\|_{L^{p_i,\kappa}}\bigg]^{p/q}=C\prod_{i=1}^m\big\|f_i\big\|_{L^{p_i,\kappa}}.
\end{split}
\end{equation*}
This shows part $(ii)$ by taking the supremum over all balls $\mathcal{B}$ in $\mathbb R^n$.
\end{itemize}
The proof of Theorem \ref{thm2} is now complete.
\end{proof}

Next we deal with the endpoint case within the framework of the Morrey space. For $1\leq p<\infty$ and $0<\kappa<1$, we now define the space $(L\log L)^{p,\kappa}(\mathbb R^n)$ as the set of all locally integrable functions $f$ on $\mathbb R^n$ such that
\begin{equation*}
\begin{split}
\big\|f\big\|_{(L\log L)^{p,\kappa}}:=&\sup_{B\subset\mathbb R^n}\frac{1}{m(B)^{\kappa/p}}\big\|f\big\|_{L^p\log L(B)}<+\infty.
\end{split}
\end{equation*}
When $p=1$, we write $(L\log L)^{1,\kappa}(\mathbb R^n)$ and $\|\cdot\|_{(L\log L)^{1,\kappa}}$. This new space was defined and investigated in \cite{sa} and \cite{lida}. By definition, it is obvious that for $(p,\kappa)\in[1,\infty)\times(0,1)$, $(L\log L)^{p,\kappa}(\mathbb R^n)\subset L^{p,\kappa}(\mathbb R^n)$, and
\begin{equation}\label{lplog2}
\big\|f\big\|_{L^{p,\kappa}}\leq \big\|f\big\|_{(L\log L)^{p,\kappa}}.
\end{equation}
In addition, it is easy to verify that for any $1\leq p<\infty$ and $0<\kappa<1$,
\begin{equation}\label{pp2}
\big\||f|^p\big\|_{(L\log L)^{1,\kappa}}\leq \big\|f\big\|^p_{(L\log L)^{p,\kappa}}.
\end{equation}
Recall that the following statement is true:
\begin{equation}\label{Mpp2}
\big\|M(f)\big\|_{L^{1,\kappa}}\leq C\big\|f\big\|_{(L\log L)^{1,\kappa}},
\end{equation}
for all $f\in (L\log L)^{1,\kappa}(\mathbb R^n)$ with $0<\kappa<1$ (see \cite[Proposition 6.1]{lida} and \cite[Corollary 2.21]{sa}). Based on the results mentioned above, we can prove the following result for the endpoint case $\min\{p_1,p_2,\dots,p_m\}=s'$.

\begin{thm}\label{thm8}
Let $\vec{\Omega}=(\Omega_1,\Omega_2,\dots,\Omega_m)\in L^s(\mathbf{S}^{n-1})$ with $1<s\leq\infty$. Let $0<\alpha<mn$, $0<\kappa<1-{(\alpha p)}/n$
\begin{equation*}
1/p=1/{p_1}+1/{p_2}+\cdots+1/{p_m}\quad \mathrm{and}\quad 1/q=1/p-\alpha/{[n(1-\kappa)]}.
\end{equation*}
If $s'=\min\{p_1,p_2,\dots,p_m\}$ and $p<n/{\alpha}$, then the inequality
\begin{equation*}
\big\|\mathcal{T}_{\Omega,\alpha;m}(\vec{f})\big\|_{L^{q,\kappa}}\lesssim
\prod_{i=1}^{\ell}\big\|f_i\big\|_{(L\log L)^{p_i,\kappa}}\cdot\prod_{j=\ell+1}^m\big\|f_j\big\|_{L^{p_j,\kappa}}
\end{equation*}
holds true. Here we suppose that $p_1=\cdots=p_{\ell}=s'$ and $p_{\ell+1}=\cdots=p_m>s'$.
\end{thm}
\begin{proof}[Proof of Theorem $\ref{thm8}$]
Taking into account \eqref{wanghua2}, we obtain
\begin{equation*}
\begin{split}
\big\|\mathcal{T}_{\Omega,\alpha;m}(\vec{f})\big\|_{L^{q,\kappa}}
&\lesssim\sup_{\mathcal{B}\subset\mathbb R^n}
\bigg(\frac{1}{m(\mathcal{B})^{\kappa}}\int_{\mathcal{B}}\big|\mathcal{M}_{s'}(\vec{f})(x)\big|^{p}dx\bigg)^{1/q} \cdot\Big[\prod_{i=1}^m\big\|f_i\big\|_{L^{p_i,\kappa}}\Big]^{1-p/q}\\
&=\sup_{\mathcal{B}\subset\mathbb R^n}\bigg(\frac{1}{m(\mathcal{B})^{\kappa}}
\int_{\mathcal{B}}\big|\mathcal{M}(\vec{f})^{s'}(x)\big|^{p/{s'}}\,dx\bigg)^{1/{q}}
\cdot\Big[\prod_{i=1}^m\big\|f_i\big\|_{L^{p_i,\kappa}}\Big]^{1-p/q}.
\end{split}
\end{equation*}
When $p_i>s'$ for some $1\leq i\leq m$ and $0<\kappa<1$, then by the strong-type boundedness of $M$ on Morrey spaces (see Lemma \ref{MM}) and \eqref{fcase}, we get
\begin{equation}\label{wang834}
\big\|M(|f_i|^{s'})\big\|_{L^{p_i/{s'},\kappa}}\leq C\big\||f_i|^{s'}\big\|_{L^{p_i/{s'},\kappa}}=C\big\|f_i\big\|_{L^{p_i,\kappa}}^{s'}.
\end{equation}
On the other hand, when $p_i=s'$ for some $1\leq i\leq m$, it then follows from \eqref{pp2} and \eqref{Mpp2} that
\begin{equation}\label{wang835}
\begin{split}
\big\|M(|f_i|^{s'})\big\|_{L^{p_i/{s'},\kappa}}&=\big\|M(|f_i|^{p_i})\big\|_{L^{1,\kappa}}\\
&\leq C\big\||f_i|^{p_i}\big\|_{(L\log L)^{1,\kappa}}\leq C\big\|f_i\big\|_{(L\log L)^{p_i,\kappa}}^{p_i}.
\end{split}
\end{equation}
Note that
\begin{equation*}
\frac{1}{p/{s'}}=\frac{1}{p_1/{s'}}+\frac{1}{p_2/{s'}}+\cdots+\frac{1}{p_m/{s'}}.
\end{equation*}
When $\min\{p_1,p_2,\dots,p_m\}=s'$, without loss of generality, we may assume that
\begin{equation*}
p_1=p_2=\cdots=p_{\ell}=s',\quad \mbox{and} \quad p_{\ell+1}=p_{\ell+2}=\cdots=p_m>s'.
\end{equation*}
As we have already shown in Section \ref{sec2}, the following pointwise estimate holds.
\begin{equation*}
\mathcal{M}(f_1^{s'},f_2^{s'},\dots,f_m^{s'})(x)\leq \prod_{i=1}^m M(|f_i|^{s'})(x),\quad \forall~x\in\mathbb R^n.
\end{equation*}
Hence, it is concluded from the inequality \eqref{multimowang}(can be viewed as H\"older's inequality for Morrey spaces) that
\begin{equation*}
\begin{split}
\big\|\mathcal{T}_{\Omega,\alpha;m}(\vec{f})\big\|_{L^{q,\kappa}}
&\lesssim\Big[\Big\|\mathcal{M}(f_1^{s'},f_2^{s'},\dots,f_m^{s'})\Big\|_{L^{p/{s'},\kappa}}\Big]^{(p/{s'})\cdot(1/q)}
\cdot\Big[\prod_{i=1}^m\big\|f_i\big\|_{L^{p_i,\kappa}}\Big]^{1-p/q}\\
&\leq\Big[\Big\|\prod_{i=1}^m M(|f_i|^{s'})\Big\|_{L^{p/{s'},\kappa}}\Big]^{(p/{s'})\cdot(1/q)}
\cdot\Big[\prod_{i=1}^m\big\|f_i\big\|_{L^{p_i,\kappa}}\Big]^{1-p/q}\\
&\leq\Big[\prod_{i=1}^m\big\|M(|f_i|^{s'})\big\|_{L^{p_i/{s'},\kappa}}\Big]^{(p/{s'})\cdot(1/q)}
\cdot\Big[\prod_{i=1}^m\big\|f_i\big\|_{L^{p_i,\kappa}}\Big]^{1-p/q}.\\
\end{split}
\end{equation*}
Furthermore, by using the estimates \eqref{lplog2}, \eqref{wang834} and \eqref{wang835}, we finally obtain
\begin{equation*}
\begin{split}
\big\|\mathcal{T}_{\Omega,\alpha;m}(\vec{f})\big\|_{L^{q,\kappa}}
&\lesssim\Big[\prod_{i=1}^{\ell}\big\|f_i\big\|^{s'}_{(L\log L)^{p_i,\kappa}}
\cdot\prod_{j=\ell+1}^m\big\|f_j\big\|_{L^{p_j,\kappa}}^{s'}\Big]^{(p/{s'})\cdot(1/q)}\\
&\times\Big[\prod_{i=1}^m\big\|f_i\big\|_{L^{p_i,\kappa}}\Big]^{1-p/q}\\
&\leq \prod_{i=1}^{\ell}\big\|f_i\big\|_{(L\log L)^{p_i,\kappa}}\cdot\prod_{j=\ell+1}^m\big\|f_j\big\|_{L^{p_j,\kappa}}.
\end{split}
\end{equation*}
This completes the proof of Theorem \ref{thm8}.
\end{proof}

\section{The critical case $\kappa=1-{(\alpha p)}/n$}

We have already shown that $\mathcal{M}_{\Omega,\alpha;m}(\vec{f})$ is controlled by $\mathcal{T}_{|\Omega|,\alpha;m}(\vec{f})$ for any $\vec{f}$.
Thus, as a direct consequence of Theorems \ref{thm2} and \ref{thm8}, we can obtain the corresponding results for the multilinear fractional maximal operator $\mathcal{M}_{\Omega,\alpha;m}$ on Morrey spaces.

\begin{thm}\label{maximalthm4}
Let $\vec{\Omega}=(\Omega_1,\Omega_2,\dots,\Omega_m)\in L^s(\mathbf{S}^{n-1})$ with $1<s\leq\infty$. Let $0<\alpha<mn$, $s'\leq p_1,p_2,\dots,p_m<\infty$, $0<\kappa<1-{(\alpha p)}/n$,
\begin{equation*}
1/p=1/{p_1}+1/{p_2}+\cdots+1/{p_m}\quad \mathrm{and}\quad 1/q=1/p-\alpha/{[n(1-\kappa)]}.
\end{equation*}

$(i)$ If $s'<p_1,p_2,\dots,p_m<\infty$, then for every $p\in(s'/m,n/{\alpha})$, the following inequality
\begin{equation*}
\big\|\mathcal{M}_{\Omega,\alpha;m}(\vec{f})\big\|_{L^{q,\kappa}}\lesssim\prod_{i=1}^m\big\|f_i\big\|_{L^{p_i,\kappa}}
\end{equation*}
holds for all $\vec{f}=(f_1,f_2,\dots,f_m)\in L^{p_1,\kappa}(\mathbb R^n)\times L^{p_2,\kappa}(\mathbb R^n)\times \cdots\times L^{p_m,\kappa}(\mathbb R^n)$.

$(ii)$ If $s'=\min\{p_1,p_2,\dots,p_m\}$, then for every $p\in[s'/m,n/{\alpha})$, the following inequality
\begin{equation*}
\big\|\mathcal{M}_{\Omega,\alpha;m}(\vec{f})\big\|_{WL^{q,\kappa}}\lesssim\prod_{i=1}^m\big\|f_i\big\|_{L^{p_i,\kappa}}
\end{equation*}
holds for all $\vec{f}=(f_1,f_2,\dots,f_m)\in L^{p_1,\kappa}(\mathbb R^n)\times L^{p_2,\kappa}(\mathbb R^n)\times \cdots\times L^{p_m,\kappa}(\mathbb R^n)$.

$(iii)$ If $s'=\min\{p_1,p_2,\dots,p_m\}$ and $p<n/{\alpha}$, then the inequality
\begin{equation*}
\big\|\mathcal{M}_{\Omega,\alpha;m}(\vec{f})\big\|_{L^{q,\kappa}}\lesssim
\prod_{i=1}^{\ell}\big\|f_i\big\|_{(L\log L)^{p_i,\kappa}}\cdot\prod_{j=\ell+1}^m\big\|f_j\big\|_{L^{p_j,\kappa}}
\end{equation*}
holds true. Here we suppose that $p_1=\cdots=p_{\ell}=s'$ and $p_{\ell+1}=\cdots=p_m>s'$.
\end{thm}

For the critical case $\kappa=1-{(\alpha p)}/n$, by using the same procedure as in the proof of Theorem \ref{maximalthm2}, we are able to show that
the multilinear fractional maximal operator $\mathcal{M}_{\Omega,\alpha;m}$ is bounded from $L^{p_1,\kappa}(\mathbb R^n)\times L^{p_2,\kappa}(\mathbb R^n)\times\cdots\times L^{p_m,\kappa}(\mathbb R^n)$ into $L^{\infty}(\mathbb R^n)$.

\begin{thm}\label{maximalthm5}
Let $\vec{\Omega}=(\Omega_1,\Omega_2,\dots,\Omega_m)\in L^s(\mathbf{S}^{n-1})$ with $1<s\leq\infty$. Suppose that
\begin{equation*}
0<\alpha<mn,\qquad  s'\leq p_1,p_2,\dots,p_m<\infty,
\end{equation*}
and
\begin{equation*}
1/p=1/{p_1}+1/{p_2}+\cdots+1/{p_m},\quad 1/q=1/p-{\alpha}/n.
\end{equation*}
If $\kappa=p/q=1-{(\alpha p)}/n$, then the following inequality
\begin{equation*}
\big\|\mathcal{M}_{\Omega,\alpha;m}(\vec{f})\big\|_{L^{\infty}}\lesssim\prod_{i=1}^m\big\|f_i\big\|_{L^{p_i,\kappa}}
\end{equation*}
holds for all $\vec{f}=(f_1,f_2,\dots,f_m)\in L^{p_1,\kappa}(\mathbb R^n)\times L^{p_2,\kappa}(\mathbb R^n)
\times\cdots\times L^{p_m,\kappa}(\mathbb R^n)$.
\end{thm}

\begin{proof}[Proof of Theorem $\ref{maximalthm5}$]
For any given ball $B(x,r)\subset\mathbb R^n$ with center $x\in\mathbb R^n$ and radius $r\in(0,\infty)$, by using H\"{o}lder's inequality and \eqref{omega88}, we get
\begin{equation*}
\begin{split}
&\prod_{i=1}^m\frac{1}{m(B(x,r))^{1-\alpha/{mn}}}\int_{|x-y_i|<r}|\Omega_i(x-y_i)|\cdot|f_i(y_i)|\,dy_i\\
&\leq\prod_{i=1}^m\frac{1}{m(B(x,r))^{1-\alpha/{mn}}}\bigg(\int_{|x-y_i|<r}|\Omega_i(x-y_i)|^sdy_i\bigg)^{1/s}
\bigg(\int_{|x-y_i|<r}|f_i(y_i)|^{s'}dy_i\bigg)^{1/{s'}}\\
&\lesssim\big\|\vec{\Omega}\big\|_{L^s(\mathbf{S}^{n-1})}
\prod_{i=1}^m\frac{1}{m(B(x,r))^{1-\alpha/{mn}}}
\cdot m(B(x,r))^{1/s}\bigg(\int_{B(x,r)}|f_i(y_i)|^{s'}dy_i\bigg)^{1/{s'}}.
\end{split}
\end{equation*}
We now consider the following two cases for the index $p_i$.

\textbf{Case 1:} $p_i=s'$ for some $1\leq i\leq m$. In this case, one has
\begin{equation*}
m(B(x,r))^{1/s}\bigg(\int_{B(x,r)}|f_i(y_i)|^{s'}dy_i\bigg)^{1/{s'}}
\leq \big\|f_i\big\|_{L^{p_i,\kappa}}\Big[m(B(x,r))\Big]^{1-{(1-\kappa)}/{p_i}}.
\end{equation*}

\textbf{Case 2:} $p_i>s'$ for some $1\leq i\leq m$. A straightforward computation leads to that
\begin{equation*}
\begin{split}
\frac{1}{s'({p_i}/{s'})'}=\frac{1}{s'}\cdot\frac{p_i-s'}{p_i}=\frac{1}{s'}-\frac{1}{p_i}.
\end{split}
\end{equation*}
This, together with H\"{o}lder's inequality, implies that
\begin{equation*}
\begin{split}
\bigg(\int_{B(x,r)}|f_i(y_i)|^{s'}dy_i\bigg)^{1/{s'}}
&\leq\bigg(\int_{B(x,r)}|f_i(y_i)|^{p_i}\,dy_i\bigg)^{1/{p_i}}\Big[m(B(x,r))\Big]^{\frac{1}{s'({p_i}/{s'})'}}\\
&=\bigg(\int_{B(x,r)}|f_i(y_i)|^{p_i}\,dy_i\bigg)^{1/{p_i}}\Big[m(B(x,r))\Big]^{1/{s'}-1/{p_i}}\\
&\leq\big\|f_i\big\|_{L^{p_i,\kappa}}\Big[m(B(x,r))\Big]^{1/{s'}-{(1-\kappa)}/{p_i}},
\end{split}
\end{equation*}
and hence
\begin{equation*}
m(B(x,r))^{1/s}\bigg(\int_{B(x,r)}|f_i(y_i)|^{s'}dy_i\bigg)^{1/{s'}}
\leq \big\|f_i\big\|_{L^{p_i,\kappa}}\Big[m(B(x,r))\Big]^{1-{(1-\kappa)}/{p_i}}.
\end{equation*}
Summing up the above estimates, we conclude that for any $x\in\mathbb R^n$ and $r>0$,
\begin{equation*}
\begin{split}
&\prod_{i=1}^m\frac{1}{m(B(x,r))^{1-\alpha/{mn}}}\int_{|x-y_i|<r}|\Omega_i(x-y_i)|\cdot|f_i(y_i)|\,dy_i\\
&\lesssim\big\|\vec{\Omega}\big\|_{L^s(\mathbf{S}^{n-1})}
\prod_{i=1}^m\frac{1}{m(B(x,r))^{1-\alpha/{mn}}}\times m(B(x,r))^{1-{(1-\kappa)}/{p_i}}\big\|f_i\big\|_{L^{p_i,\kappa}}\\
&=\big\|\vec{\Omega}\big\|_{L^s(\mathbf{S}^{n-1})}\prod_{i=1}^m\big\|f_i\big\|_{L^{p_i,\kappa}}
\frac{1}{m(B(x,r))^{m-\alpha/n}}\cdot m(B(x,r))^{\sum_{i=1}^m[1-{(1-\kappa)}/{p_i}]}\\
&=\big\|\vec{\Omega}\big\|_{L^s(\mathbf{S}^{n-1})}\prod_{i=1}^m\big\|f_i\big\|_{L^{p_i,\kappa}},
\end{split}
\end{equation*}
where in the last equation we have used the fact that
\begin{equation*}
\sum_{i=1}^m {(1-\kappa)}/{p_i}={(1-\kappa)}/p=\alpha/n.
\end{equation*}
The desired result follows by taking the supremum for all $r>0$. We are done.
\end{proof}

As for the multilinear fractional integral operator $\mathcal{T}_{\Omega,\alpha;m}$, the following result can also be proved with an additional assumption.
\begin{thm}\label{maximalthm6}
Let $0<\alpha<mn$, $1\leq p_1,p_2,\dots,p_m<\infty,$ and
\begin{equation*}
1/p=1/{p_1}+1/{p_2}+\cdots+1/{p_m},\quad 1/q=1/p-{\alpha}/n.
\end{equation*}
Suppose that $\Omega_1=\Omega_2=\cdots=\Omega_m\equiv1$ and $\kappa=p/q=1-{(\alpha p)}/n$, then the following inequality
\begin{equation*}
\big\|\mathcal{T}_{\Omega,\alpha;m}(\vec{f})\big\|_{\mathrm{BMO}}\lesssim\prod_{i=1}^m\big\|f_i\big\|_{L^{p_i,\kappa}}
\end{equation*}
holds for all $\vec{f}=(f_1,f_2,\dots,f_m)\in L^{p_1,\kappa}(\mathbb R^n)\times L^{p_2,\kappa}(\mathbb R^n)
\times\cdots\times L^{p_m,\kappa}(\mathbb R^n)$.
\end{thm}
This result has been obtained by Tang in \cite{tang}. Actually, we can remove the restriction of $(m-1)n<\alpha<mn$ in \cite{tang}, and prove that the same conclusion also holds for $0<\alpha<mn$ with only minor modifications.

We now claim that the inclusion relation
\begin{equation*}
L^{p,\infty}(\mathbb R^n)\subset L^{q,\kappa}(\mathbb R^n)
\end{equation*}
holds with $1\leq q<p$ and $\kappa=1-q/p$. Moreover, for any $f\in L^{p,\infty}(\mathbb R^n)$ with $1<p<\infty$,  there exists a constant $C=C(p,q)>0$ such that
\begin{equation}\label{cpq}
\big\|f\big\|_{L^{q,\kappa}}\leq C\big\|f\big\|_{L^{p,\infty}}
\end{equation}
holds with $1\leq q<p$ and $\kappa=1-q/p$. In fact, for any ball $\mathcal{B}$ in $\mathbb R^n$, we have
\begin{equation*}
\begin{split}
\big\|f\cdot\chi_{\mathcal{B}}\big\|_{L^q}&=\bigg(\int_{\mathcal{B}}|f(x)|^q\,dx\bigg)^{1/q}\\
&=\bigg(\int_0^{\infty}q\lambda^{q-1}m\big(\big\{x\in \mathcal{B}:|f(x)|>\lambda\big\}\big)\,d\lambda\bigg)^{1/q}\\
&=\bigg(\int_{0}^{\sigma}\cdots+\int_{\sigma}^{\infty}\cdots\bigg)^{1/q},
\end{split}
\end{equation*}
where $\sigma>0$ is a constant to be chosen later. Clearly, the first integral in the bracket is bounded by
\begin{equation*}
\begin{split}
\int_0^{\sigma}q\lambda^{q-1}m(\mathcal{B})\,d\lambda=\sigma^q\cdot m(\mathcal{B}).
\end{split}
\end{equation*}
Using the fact that $q-p<0$, the second integral in the bracket is then controlled by
\begin{equation*}
\begin{split}
&\int_{\sigma}^{\infty}q\lambda^{q-1}m\big(\big\{x\in\mathbb R^n:|f(x)|>\lambda\big\}\big)\,d\lambda\\
&\leq\int_{\sigma}^{\infty}q\lambda^{q-1}\Big(\frac{\|f\|_{L^{p,\infty}}}{\lambda}\Big)^pd\lambda\\
&=\frac{q}{p-q}\Big(\sigma^{q-p}\cdot\big\|f\big\|^p_{L^{p,\infty}}\Big).
\end{split}
\end{equation*}
We now choose $\sigma>0$ such that
\begin{equation*}
\sigma^q\cdot m(\mathcal{B})=\sigma^{q-p}\cdot\big\|f\big\|^p_{L^{p,\infty}}.
\end{equation*}
That is,
\begin{equation*}
\sigma:=\frac{\|f\|_{L^{p,\infty}}}{m(\mathcal{B})^{1/p}}.
\end{equation*}
Hence,
\begin{equation*}
\begin{split}
\big\|f\cdot\chi_{\mathcal{B}}\big\|_{L^q}
&\leq\Big[1+\frac{q}{p-q}\Big]^{1/q}\Big[\sigma^q\cdot m(\mathcal{B})\Big]^{1/q}\\
&=\Big[\frac{p}{p-q}\Big]^{1/q}m(\mathcal{B})^{1/q-1/p}\big\|f\big\|_{L^{p,\infty}},
\end{split}
\end{equation*}
which is equivalent to that
\begin{equation*}
\frac{1}{m(\mathcal{B})^{1/q-1/p}}\big\|f\cdot\chi_{\mathcal{B}}\big\|_{L^q}
\leq \Big[\frac{p}{p-q}\Big]^{1/q}\big\|f\big\|_{L^{p,\infty}}.
\end{equation*}
This gives the desired estimate \eqref{cpq} and completes the proof of the claim, by taking the supremum over all balls $\mathcal{B}$ in $\mathbb R^n$.

As an immediate consequence of \eqref{cpq}, we obtain that for any given $\vec{f}=(f_1,f_2,\dots,f_m)\in L^{p^{*}_1,\infty}(\mathbb R^n)\times L^{p^{*}_2,\infty}(\mathbb R^n)\times\cdots\times L^{p^{*}_m,\infty}(\mathbb R^n)$ with
\begin{equation*}
s'<p^{*}_1,p^{*}_2,\dots,p^{*}_m<\infty\quad \mbox{and} \quad 1<s\leq\infty,
\end{equation*}
the inequalities
\begin{equation}\label{cpq2}
\big\|f_i\big\|_{L^{p_i,\kappa}}\leq C\big\|f_i\big\|_{L^{p^{*}_i,\infty}},\quad i=1,2,\dots,m
\end{equation}
hold, whenever $s'\leq p_i<p^{*}_i$ and $\kappa=1-{p_i}/{p^{*}_i}$. Suppose that $0<\alpha<mn$, $0<p<n/{\alpha}$,
\begin{equation*}
\frac{\alpha}{\,n\,}=\frac{1}{p^{*}_1}+\frac{1}{p^{*}_2}+\cdots+\frac{1}{p^{*}_m}\quad \&
\quad \frac{1}{\,p\,}=\frac{1}{p_1}+\frac{1}{p_2}+\cdots+\frac{1}{p_m}.
\end{equation*}
By a simple calculation, one has
\begin{equation}\label{cpq3}
\begin{split}
&\frac{\kappa}{\,p\,}=\kappa\cdot\sum_{i=1}^m\frac{1}{p_i}
=\sum_{i=1}^m\Big(\frac{1}{p_i}-\frac{1}{p^{*}_i}\Big)=\frac{1}{\,p\,}-\frac{\alpha}{\,n\,}.\\
\Longrightarrow &\kappa=1-\frac{\alpha p}{n}.
\end{split}
\end{equation}

Therefore, in view of \eqref{cpq2} and \eqref{cpq3}, we have the following results.

\begin{cor}\label{cor16}
Let $\vec{\Omega}=(\Omega_1,\Omega_2,\dots,\Omega_m)\in L^s(\mathbf{S}^{n-1})$ with $1<s\leq\infty$. Suppose that
\begin{equation*}
0<\alpha<mn,\qquad  s'<p_1,p_2,\dots,p_m<\infty,
\end{equation*}
and
\begin{equation*}
{\alpha}/n=1/{p_1}+1/{p_2}+\cdots+1/{p_m}.
\end{equation*}
Then the following inequality
\begin{equation*}
\big\|\mathcal{M}_{\Omega,\alpha;m}(\vec{f})\big\|_{L^{\infty}}\lesssim\prod_{i=1}^m\big\|f_i\big\|_{L^{p_i,\infty}}
\end{equation*}
holds for all $\vec{f}=(f_1,f_2,\dots,f_m)\in L^{p_1,\infty}(\mathbb R^n)\times L^{p_2,\infty}(\mathbb R^n)
\times\cdots\times L^{p_m,\infty}(\mathbb R^n)$.
\end{cor}

\begin{cor}\label{cor18}
Let $0<\alpha<mn$, $1<p_1,p_2,\dots,p_m<\infty,$ and
\begin{equation*}
{\alpha}/n=1/{p_1}+1/{p_2}+\cdots+1/{p_m}.
\end{equation*}
Suppose that $\Omega_1=\Omega_2=\cdots=\Omega_m\equiv1$. Then the following inequality
\begin{equation*}
\big\|\mathcal{T}_{\Omega,\alpha;m}(\vec{f})\big\|_{\mathrm{BMO}}\lesssim\prod_{i=1}^m\big\|f_i\big\|_{L^{p_i,\infty}}
\end{equation*}
holds for all $\vec{f}=(f_1,f_2,\dots,f_m)\in L^{p_1,\infty}(\mathbb R^n)\times L^{p_2,\infty}(\mathbb R^n)
\times\cdots\times L^{p_m,\infty}(\mathbb R^n)$.
\end{cor}

Finally, we consider the (\textbf{non-centered}) multilinear fractional maximal operator $\mathcal{M}'_{\Omega,\alpha;m}$, which is given by
\begin{equation*}
\mathcal{M}'_{\Omega,\alpha;m}(\vec{f})(x):=\sup_{B\ni x}\prod_{i=1}^m
\frac{1}{m(B)^{1-\alpha/{mn}}}\int_{B}\big|\Omega_i(x-y_i)f_i(y_i)\big|\,dy_i,
\end{equation*}
where the supremum is taken over all balls $B$ containing the point $x$ in $\mathbb R^n$. By using similar arguments, we can see that all the results derived above are also true for the multilinear operator $\mathcal{M}'_{\Omega,\alpha;m}$, when $\vec{\Omega}=(\Omega_1,\Omega_2,\dots,\Omega_m)\in L^s(\mathbf{S}^{n-1})$ with $1<s\leq\infty$ and $0<\alpha<mn$ with $m\geq2$. The details are omitted here.

\section{Applications to several integral inequalities on $\mathbb R^n$}
\subsection{Hardy--Littlewood--Sobolev inequalities}
The classical Hardy--Littlewood--Sobolev inequality on $\mathbb R^n$ says that for any $0<\lambda<n$ and all measurable functions
$(f,g)\in L^p(\mathbb R^n)\times L^q(\mathbb R^n)$, we have
\begin{equation}\label{HLS}
\bigg|\int_{\mathbb R^n}\int_{\mathbb R^n}\frac{f(x)\cdot g(y)}{|x-y|^{\lambda}}\,dxdy\bigg|\leq C({p,q,\lambda,n})\|f\|_{L^p}\cdot\|g\|_{L^q},
\end{equation}
whenever $1<p,q<\infty$ and $1/p+1/q+\lambda/n=2$. This inequality appears in many areas of analysis(as a special case of Young's inequality), often in their dual forms as Sobolev inequalities. The sharp version of inequality \eqref{HLS} with optimal constant was proved by Lieb in \cite{lieb}. It was shown in \cite{lieb} that if $p=q={2n}/{(2n-\lambda)}$ and $0<\lambda<n$, then there exists a sharp constant $C(n,\lambda)$, independent of $f$ and $g$, such that (see also \cite{li})
\begin{equation}\label{HLSsharp}
\bigg|\int_{\mathbb R^n}\int_{\mathbb R^n}\frac{f(x)\cdot g(y)}{|x-y|^{\lambda}}\,dxdy\bigg|\leq C({n,\lambda})\big\|f\big\|_{L^{\frac{2n}{2n-\lambda}}}\cdot\big\|g\big\|_{L^{\frac{2n}{2n-\lambda}}},
\end{equation}
where the sharp constant $C({n,\lambda})$ is given by
\begin{equation*}
C({n,\lambda}):=\pi^{\lambda/2}\cdot\frac{\Gamma(n/2-\lambda/2)}{\Gamma(n-\lambda/2)}
\left(\frac{\Gamma(n/2)}{\Gamma(n)}\right)^{-1+\lambda/n}.
\end{equation*}
In this case, the equality in \eqref{HLSsharp} holds if and only if $f$ and $g$ can be written as $\mathcal{A}(\gamma^2+|x-x_0|^2)^{{(\lambda-2n)}/2}$ for some $\mathcal{A}\in \mathbb{C}$, $0\neq\gamma\in\mathbb R$ and $x_0\in\mathbb R^n$.
For more information about the sharp constants and the optimizers in \eqref{HLS}, as well as some generalizations, the reader is referred to \cite{li} and \cite{lieb2}.

By the results shown in Section \ref{sec2} (in the linear case), we can extend \eqref{HLS} to the following:
\begin{thm}\label{31}
Suppose that $\Omega\in L^s(\mathbf{S}^{n-1})$ with $1<s\leq\infty$. Let $0<\lambda<n$, $s'<p,q<\infty$ and $1/p+1/q+\lambda/n=2$. Then for any $f\in L^p(\mathbb R^n)$ and $g\in L^q(\mathbb R^n)$, there exists a constant $C>0$ independent of $f$ and $g$ such that
\begin{equation}\label{HLSWang}
\bigg|\int_{\mathbb R^n}\int_{\mathbb R^n}\frac{\Omega(x-y)f(x)\cdot g(y)}{|x-y|^{\lambda}}\,dxdy\bigg|
\leq C\|f\|_{L^p}\cdot\|g\|_{L^q}.
\end{equation}
\end{thm}

\begin{proof}
Using H\"{o}lder's inequality, we have
\begin{equation*}
\begin{split}
&\bigg|\int_{\mathbb R^n}\int_{\mathbb R^n}\frac{\Omega(x-y)f(x)\cdot g(y)}{|x-y|^{\lambda}}\,dxdy\bigg|\\
&=\bigg|\int_{\mathbb R^n}f(x)\bigg(\int_{\mathbb R^n}\frac{\Omega(x-y)g(y)}{|x-y|^{\lambda}}\,dy\bigg)dx\bigg|\\
&=\bigg|\int_{\mathbb R^n}f(x)\cdot T_{\Omega,n-\lambda}(g)(x)\,dx\bigg|\\
&\leq\big\|f\big\|_{L^p}\cdot\big\|T_{\Omega,n-\lambda}(g)\big\|_{L^{p'}}.
\end{split}
\end{equation*}
Observe that $1/{p'}=1/q-{(n-\lambda)}/n$ and $s'<q<n/{(n-\lambda)}$. Applying part $(i)$ of Theorem \ref{thm1}(for the linear case), we get
\begin{equation*}
\big\|T_{\Omega,n-\lambda}(g)\big\|_{L^{p'}}\leq C\big\|g\big\|_{L^q}.
\end{equation*}
Changing the order of integration yields
\begin{equation*}
\int_{\mathbb R^n}f(x)\cdot T_{\Omega,n-\lambda}(g)(x)\,dx=\int_{\mathbb R^n}T_{\widetilde{\Omega},n-\lambda}(f)(y)\cdot g(y)\,dy,
\end{equation*}
where $\widetilde{\Omega}(x):=\Omega(-x)$. Obviously, under the same assumptions as in Theorem \ref{thm1}, the conclusion of Theorem \ref{thm1} (for the linear case) also holds for $\widetilde{\Omega}(x)$. This fact along with H\"{o}lder's inequality yields
\begin{equation*}
\begin{split}
\bigg|\int_{\mathbb R^n}f(x)\cdot T_{\Omega,n-\lambda}(g)(x)\,dx\bigg|
&=\bigg|\int_{\mathbb R^n}T_{\widetilde{\Omega},n-\lambda}(f)(y)\cdot g(y)\,dy\bigg|\\
&\leq \big\|T_{\widetilde{\Omega},n-\lambda}(f)\big\|_{L^{q'}}\big\|g\big\|_{L^q}\\
&\leq C\big\|f\big\|_{L^p}\cdot\big\|g\big\|_{L^q},
\end{split}
\end{equation*}
where in the last inequality we have used the fact that $1/{q'}=1/p-{(n-\lambda)}/n$ and $s'<p<n/{(n-\lambda)}$. From this, the desired inequality follows.
\end{proof}

Observe that \eqref{HLSWang} can be rewritten as
\begin{equation*}
\big\|f\cdot T_{\Omega,n-\lambda}(g)\big\|_{L^1}\leq C\big\|f\big\|_{L^p}\cdot\big\|g\big\|_{L^q},
\end{equation*}
for any $f\in L^p(\mathbb R^n)$ and $g\in L^q(\mathbb R^n)$ with $s'<p,q<\infty$, when $0<\lambda<n$ and $1/p+1/q+\lambda/n=2$. Equivalently, when $0<\alpha<n$ and $1/p+1/q=1+\alpha/n$, we know that
\begin{equation}\label{L1norm}
\big\|f\cdot T_{\Omega,\alpha}(g)\big\|_{L^1}\leq C\big\|f\big\|_{L^p}\cdot\big\|g\big\|_{L^q},
\end{equation}
for any $f\in L^p(\mathbb R^n)$ and $g\in L^q(\mathbb R^n)$ with $s'<p,q<\infty$.

Motivated by the estimate \eqref{L1norm}, let us now consider the multilinear version of the HLS inequality. Recall that the following two estimates hold (see \cite[p.11 and p.16]{grafakos}).

\begin{lem}\label{weakHold}
Let $1\leq p,q<\infty$ and $0<r<\infty$ such that $1/r=1/p+1/q$. Then we have
\begin{enumerate}
  \item H\"older's inequality for $L^p$ spaces:
  \begin{equation*}
  \|\mathcal{F}\cdot \mathcal{G}\|_{L^r}\leq \|\mathcal{F}\|_{L^p}\cdot\|\mathcal{G}\|_{L^q}.
  \end{equation*}
  \item H\"older's inequality for weak $L^p$ spaces:
  \begin{equation*}
  \|\mathcal{F}\cdot \mathcal{G}\|_{L^{r,\infty}}\leq C\|\mathcal{F}\|_{L^p}\cdot\|\mathcal{G}\|_{L^{q,\infty}},
  \end{equation*}
  where $C$ is a positive constant independent of $\mathcal{F}$ and $\mathcal{G}$.
\end{enumerate}
\end{lem}
As a straightforward consequence of Theorem \ref{thm1} and Lemma \ref{weakHold}, we can show that the $L^1$ norm in \eqref{L1norm} can be replaced by the general $L^r$ norm (or $L^{r,\infty}$ norm), and $T_{\Omega,\alpha}$ replaced by $\mathcal{T}_{\Omega,\alpha;m}$ with $m\geq2$.

\begin{thm}\label{thm15}
Suppose that $\vec{\Omega}=(\Omega_1,\Omega_2,\dots,\Omega_m)\in L^s(\mathbf{S}^{n-1})$ with $1<s\leq\infty$. Let $0<\alpha<mn$, $1\leq p<\infty$, $s'\leq p_1,p_2,\dots,p_m<\infty$ and
\begin{equation*}
1/p+\sum_{i=1}^m 1/{p_i}=1/r+\alpha/n
\end{equation*}
with $0<r<\min\{p,q\}$ and $1/q=\sum_{i=1}^m1/{p_i}-\alpha/n$.

$(i)$ If $s'<p_1,p_2,\dots,p_m<\infty$, then for any $f\in L^p(\mathbb R^n)$ and $\vec{g}=(g_1,g_2,\dots,g_m)\in L^{p_1}(\mathbb R^n)\times L^{p_2}(\mathbb R^n)\times \cdots\times L^{p_m}(\mathbb R^n)$, there exists a constant $C>0$ independent of $f$ and $\vec{g}$ such that
\begin{equation*}
\big\|f\cdot \mathcal{T}_{\Omega,\alpha;m}(\vec{g})\big\|_{L^{r}}\leq C\big\|f\big\|_{L^p}\cdot\Big(\prod_{i=1}^m\big\|g_i\big\|_{L^{p_i}}\Big).
\end{equation*}

$(ii)$ If $s'=\min\{p_1,p_2,\dots,p_m\}<\infty$, then for any $f\in L^p(\mathbb R^n)$ and $\vec{g}=(g_1,g_2,\dots,g_m)\in L^{p_1}(\mathbb R^n)\times L^{p_2}(\mathbb R^n)\times \cdots\times L^{p_m}(\mathbb R^n)$, there exists a constant $C>0$ independent of $f$ and $\vec{g}$ such that
\begin{equation*}
\big\|f\cdot \mathcal{T}_{\Omega,\alpha;m}(\vec{g})\big\|_{L^{r,\infty}}\leq C\big\|f\big\|_{L^p}\cdot\Big(\prod_{i=1}^m\big\|g_i\big\|_{L^{p_i}}\Big).
\end{equation*}
\end{thm}

\begin{proof}
Using H\"{o}lder's inequality for $L^p$ spaces(part (1) of Lemma \ref{weakHold}), we see that
\begin{equation*}
\big\|f\cdot \mathcal{T}_{\Omega,\alpha;m}(\vec{g})\big\|_{L^{r}}\leq \big\|f\big\|_{L^p}\cdot\big\|\mathcal{T}_{\Omega,\alpha;m}(\vec{g})\big\|_{L^{q}},
\end{equation*}
where the positive number $q$ satisfies $1/p+1/q=1/r$, that is,
\begin{equation*}
1/q=1/r-1/p=\sum_{i=1}^m1/{p_i}-\alpha/n.
\end{equation*}
Therefore, by Theorem \ref{thm1}, we conclude that
\begin{equation*}
\big\|f\cdot \mathcal{T}_{\Omega,\alpha;m}(\vec{g})\big\|_{L^{r}}
\leq C\big\|f\big\|_{L^p}\cdot\Big(\prod_{i=1}^m\big\|g_i\big\|_{L^{p_i}}\Big).
\end{equation*}
This proves part $(i)$. The proof of part $(ii)$ is treated in the same manner.
\end{proof}

As a straightforward consequence of Theorem \ref{thm7} and Lemma \ref{weakHold}, we have the following result.

\begin{thm}\label{thm16}
Suppose that $\vec{\Omega}=(\Omega_1,\Omega_2,\dots,\Omega_m)\in L^s(\mathbf{S}^{n-1})$ with $1<s\leq\infty$. Let $0<\alpha<mn$, $1\leq p<\infty$, $s'=\min\{p_1,p_2,\dots,p_m\}$ and
\begin{equation*}
1/p+\sum_{i=1}^m 1/{p_i}=1/r+\alpha/n
\end{equation*}
with $0<r<\min\{p,q\}$ and $1/q=\sum_{i=1}^m1/{p_i}-\alpha/n$. Then for any given ball $B$ in $\mathbb R^n$, we have
\begin{equation*}
\big\|f\cdot \mathcal{T}_{\Omega,\alpha;m}(\vec{g})\big\|_{L^{r}(B)}
\leq C\big\|f\big\|_{L^p(B)}\cdot\bigg(\prod_{i=1}^{\ell}\big\|g_i\big\|_{L^{p_i}\log L(B)}\cdot
\prod_{j=\ell+1}^m\big\|g_j\big\|_{L^{p_j}(B)}\bigg).
\end{equation*}
Here we suppose that $p_1=\cdots=p_{\ell}=s'$ and $p_{\ell+1}=\cdots=p_m>s'$.
\end{thm}

\begin{proof}
It follows directly from H\"{o}lder's inequality that for any given ball $B$ in $\mathbb R^n$,
\begin{equation*}
\big\|f\cdot \mathcal{T}_{\Omega,\alpha;m}(\vec{g})\big\|_{L^{r}(B)}
\leq\big\|f\big\|_{L^p(B)}\cdot\big\|\mathcal{T}_{\Omega,\alpha;m}(\vec{g})\big\|_{L^q(B)},
\end{equation*}
where the positive number $q$ is given by $1/q=1/r-1/p=\sum_{i=1}^m1/{p_i}-\alpha/n$. Therefore, by using Theorem \ref{thm7}, we conclude that
\begin{equation*}
\big\|f\cdot \mathcal{T}_{\Omega,\alpha;m}(\vec{g})\big\|_{L^{r}(B)}
\leq C\big\|f\big\|_{L^p(B)}\cdot\bigg(\prod_{i=1}^{\ell}\big\|g_i\big\|_{L^{p_i}\log L(B)}\cdot
\prod_{j=\ell+1}^m\big\|g_j\big\|_{L^{p_j}(B)}\bigg).
\end{equation*}
This is our desired inequality.
\end{proof}

\subsection{Olsen-type inequalities}
A classical result of Olsen \cite{ol} states that
\begin{equation*}
\big\|f\cdot I_{\alpha}(g)\big\|_{L^{r,\kappa}}
\leq C({p,q,\kappa,n})\big\|f\big\|_{L^{p,\kappa}}\cdot\big\|g\big\|_{L^{q,\kappa}},
\end{equation*}
under certain conditions on the parameters $p,q,r,\kappa$. Olsen considered this type of inequality to investigate the Schr\"{o}dinger equation. Later this inequality (also known as the trace inequality) was studied and sharpened by many authors. For more related results, the reader is referred to \cite{ol}, \cite{lida}, \cite{lida2} and \cite{sa}. There is an analogue of H\"older's inequality for Morrey spaces and weak Morrey spaces. Let $0<\kappa<1$, $1\leq p,q<\infty$ and $0<r<\infty$ with $1/p+1/q=1/r$. We have already shown that
\begin{equation}\label{lastin}
\big\|\mathcal{F}\cdot \mathcal{G}\big\|_{L^{r,\kappa}}
\leq\big\|\mathcal{F}\big\|_{L^{p,\kappa}}\cdot\big\|\mathcal{G}\big\|_{L^{q,\kappa}}.
\end{equation}
It should be pointed out that the above estimate has been proved by Olsen \cite{ol} if $1\leq r<\infty$. We remark that by using similar arguments, the weak version of \eqref{lastin} also holds true. Indeed, by using H\"{o}lder's inequality for weak $L^p$ spaces, we get
\begin{equation*}
\begin{split}
\big\|\mathcal{F}\cdot \mathcal{G}\big\|_{L^{r,\infty}(B)}
&\leq C\big\|\mathcal{F}\big\|_{L^p(B)}\big\|\mathcal{G}\big\|_{L^{q,\infty}(B)}\\
&= C\big\|\mathcal{F}\cdot\chi_{B}\big\|_{L^p(B)}\big\|\mathcal{G}\cdot\chi_{B}\big\|_{L^{q,\infty}(B)}\\
&\leq C\Big(\big\|\mathcal{F}\big\|_{L^{p,\kappa}}m(B)^{\kappa/p}\Big)
\cdot\Big(\big\|\mathcal{G}\big\|_{WL^{q,\kappa}}m(B)^{\kappa/q}\Big)\\
&=C\Big(\big\|\mathcal{F}\big\|_{L^{p,\kappa}}\cdot\big\|\mathcal{G}\big\|_{WL^{q,\kappa}}\Big)m(B)^{\kappa/r},
\end{split}
\end{equation*}
which in turn implies that
\begin{equation}\label{dot}
\frac{1}{m(B)^{\kappa/r}}\big\|\mathcal{F}\cdot \mathcal{G}\big\|_{L^{r,\infty}(B)}
\leq C\Big(\big\|\mathcal{F}\big\|_{L^{p,\kappa}}\cdot\big\|\mathcal{G}\big\|_{WL^{q,\kappa}}\Big).
\end{equation}
Since \eqref{dot} is independent of the choice of the ball $B$, we conclude that $\mathcal{F}\cdot \mathcal{G}$ is in the space $WL^{r,\kappa}(\mathbb R^n)$ and
\begin{equation*}
\big\|\mathcal{F}\cdot \mathcal{G}\big\|_{WL^{r,\kappa}}
\leq C\big\|\mathcal{F}\big\|_{L^{p,\kappa}}\cdot\big\|\mathcal{G}\big\|_{WL^{q,\kappa}}.
\end{equation*}

This gives us the following:
\begin{lem}\label{weakMHold}
Let $0<\kappa<1$, $1\leq p,q<\infty$ and $0<r<\infty$ such that $1/r=1/p+1/q$. Then we have
\begin{enumerate}
  \item H\"older's inequality for Morrey spaces:
\begin{equation*}
\big\|\mathcal{F}\cdot \mathcal{G}\big\|_{L^{r,\kappa}}
\leq \big\|\mathcal{F}\big\|_{L^{p,\kappa}}\cdot\big\|\mathcal{G}\big\|_{L^{q,\kappa}}.
\end{equation*}
  \item H\"older's inequality for weak Morrey spaces:
\begin{equation*}
\big\|\mathcal{F}\cdot \mathcal{G}\big\|_{WL^{r,\kappa}}
\leq C\big\|\mathcal{F}\big\|_{L^{p,\kappa}}\cdot\big\|\mathcal{G}\big\|_{WL^{q,\kappa}}.
\end{equation*}
\end{enumerate}
\end{lem}

In view of Lemma \ref{weakMHold} and Theorems \ref{thm2} and \ref{thm8}, we can also prove the following Olsen-type inequalities for  $\mathcal{T}_{\Omega,\alpha;m}$ with $m\geq2$.

\begin{thm}\label{thm17}
Suppose that $\vec{\Omega}=(\Omega_1,\Omega_2,\dots,\Omega_m)\in L^s(\mathbf{S}^{n-1})$ with $1<s\leq\infty$. Let $0<\alpha<mn$, $1\leq p<\infty$, $s'\leq p_1,p_2,\dots,p_m<\infty$, $0<\kappa<1$ and
\begin{equation*}
1/p+\sum_{i=1}^m 1/{p_i}=1/r+\alpha/{[n(1-\kappa)]}
\end{equation*}
with $0<r<\min\{p,q\}$ and $1/q=\sum_{i=1}^m 1/{p_i}-\alpha/{[n(1-\kappa)]}$.

$(i)$ If $s'<p_1,p_2,\dots,p_m<\infty$, then for any $f\in L^{p,\kappa}(\mathbb R^n)$ and $\vec{g}=(g_1,g_2,\dots,g_m)\in L^{p_1,\kappa}(\mathbb R^n)\times L^{p_2,\kappa}(\mathbb R^n)\times \cdots\times L^{p_m,\kappa}(\mathbb R^n)$, there exists a constant $C>0$ independent of $f$ and $\vec{g}$ such that
\begin{equation*}
\big\|f\cdot \mathcal{T}_{\Omega,\alpha;m}(\vec{g})\big\|_{L^{r,\kappa}}
\leq C\big\|f\big\|_{L^{p,\kappa}}\cdot\Big(\prod_{i=1}^m\big\|g_i\big\|_{L^{p_i,\kappa}}\Big).
\end{equation*}

$(ii)$ If $s'=\min\{p_1,p_2,\dots,p_m\}<\infty$, then for any $f\in L^{p,\kappa}(\mathbb R^n)$ and $\vec{g}=(g_1,g_2,\dots,g_m)\in L^{p_1,\kappa}(\mathbb R^n)\times L^{p_2,\kappa}(\mathbb R^n)\times \cdots\times L^{p_m,\kappa}(\mathbb R^n)$, there exists a constant $C>0$ independent of $f$ and $\vec{g}$ such that
\begin{equation*}
\big\|f\cdot \mathcal{T}_{\Omega,\alpha;m}(\vec{g})\big\|_{WL^{r,\kappa}}
\leq C\big\|f\big\|_{L^{p,\kappa}}\cdot\Big(\prod_{i=1}^m\big\|g_i\big\|_{L^{p_i,\kappa}}\Big).
\end{equation*}
\end{thm}

\begin{proof}
We first observe that the condition $1/p+\sum_{i=1}^m 1/{p_i}=1/r+\alpha/{[n(1-\kappa)]}$ with $0<r<\min\{p,q\}$ implies that
\begin{equation*}
0<\kappa<1-{(\alpha p^{*})}/n\quad \& \quad 1/{p^{*}}=\sum_{i=1}^m 1/{p_i}.
\end{equation*}
When $s'<p_1,p_2,\dots,p_m<\infty$, by using part (1) of Lemma \ref{weakMHold}, then we have
\begin{equation*}
\big\|f\cdot \mathcal{T}_{\Omega,\alpha;m}(\vec{g})\big\|_{L^{r,\kappa}}
\leq \big\|f\big\|_{L^{p,\kappa}}\cdot\big\|\mathcal{T}_{\Omega,\alpha;m}(\vec{g})\big\|_{L^{q,\kappa}},
\end{equation*}
where the positive number $q$ satisfies $1/r=1/p+1/{q}$. Notice that
\begin{equation*}
1/{q}=\sum_{i=1}^m 1/{p_i}-\alpha/{[n(1-\kappa)]}=1/{p^{*}}-\alpha/{[n(1-\kappa)]}.
\end{equation*}
Thus, by Theorem \ref{thm2}, we further obtain
\begin{equation*}
\big\|f\cdot \mathcal{T}_{\Omega,\alpha;m}(\vec{g})\big\|_{L^{r,\kappa}}
\leq C\big\|f\big\|_{L^{p,\kappa}}\cdot\Big(\prod_{i=1}^m\big\|g_i\big\|_{L^{p_i,\kappa}}\Big),
\end{equation*}
whenever $0<\kappa<1-{(\alpha p^{*})}/n$. This shows $(i)$.

When $s'=\min\{p_1,p_2,\dots,p_m\}<\infty$, from part (2) of Lemma \ref{weakMHold} and Theorem \ref{thm2}, it then follows that
\begin{equation*}
\begin{split}
\big\|f\cdot \mathcal{T}_{\Omega,\alpha;m}(\vec{g})\big\|_{WL^{r,\kappa}}
&\leq C\big\|f\big\|_{L^{p,\kappa}}\cdot\big\|\mathcal{T}_{\Omega,\alpha;m}(\vec{g})\big\|_{WL^{q,\kappa}}\\
&\leq C\big\|f\big\|_{L^{p,\kappa}}\cdot\Big(\prod_{i=1}^m\big\|g_i\big\|_{L^{p_i,\kappa}}\Big).
\end{split}
\end{equation*}
This shows part $(ii)$ and the proof is complete.
\end{proof}

\begin{thm}\label{thm18}
Suppose that $\vec{\Omega}=(\Omega_1,\Omega_2,\dots,\Omega_m)\in L^s(\mathbf{S}^{n-1})$ with $1<s\leq\infty$. Let $0<\alpha<mn$, $1\leq p<\infty$, $s'=\min\{p_1,p_2,\dots,p_m\}$, $0<\kappa<1$ and
\begin{equation*}
1/p+\sum_{i=1}^m 1/{p_i}=1/r+\alpha/{[n(1-\kappa)]}
\end{equation*}
with $0<r<\min\{p,q\}$ and $1/q=\sum_{i=1}^m 1/{p_i}-\alpha/{[n(1-\kappa)]}$. Then we have
\begin{equation*}
\big\|f\cdot \mathcal{T}_{\Omega,\alpha;m}(\vec{g})\big\|_{L^{r,\kappa}}
\leq C\big\|f\big\|_{L^{p,\kappa}}\cdot\bigg(\prod_{i=1}^{\ell}\big\|g_i\big\|_{(L\log L)^{p_i,\kappa}}\cdot
\prod_{j=\ell+1}^m\big\|g_j\big\|_{L^{p_j,\kappa}}\bigg).
\end{equation*}
Here we suppose that $p_1=\cdots=p_{\ell}=s'$ and $p_{\ell+1}=\cdots=p_m>s'$.
\end{thm}

\begin{proof}
It follows immediately from part (1) of Lemma \ref{weakMHold} that
\begin{equation*}
\big\|f\cdot \mathcal{T}_{\Omega,\alpha;m}(\vec{g})\big\|_{L^{r,\kappa}}
\leq \big\|f\big\|_{L^{p,\kappa}}\cdot\big\|\mathcal{T}_{\Omega,\alpha;m}(\vec{g})\big\|_{L^{q,\kappa}},
\end{equation*}
where $1/r=1/p+1/{q}$. Note that
\begin{equation*}
1/q=\sum_{i=1}^m 1/{p_i}-\alpha/{[n(1-\kappa)]}=1/{p^{*}}-\alpha/{[n(1-\kappa)]}\quad \& \quad 0<\kappa<1-{(\alpha p^{*})}/n
\end{equation*}
with $1/{p^*}=\sum_{i=1}^m 1/{p_i}$. Thus, by using Theorem \ref{thm8}, we deduce that
\begin{equation*}
\big\|f\cdot \mathcal{T}_{\Omega,\alpha;m}(\vec{g})\big\|_{L^{r,\kappa}}
\leq C\big\|f\big\|_{L^{p,\kappa}}\cdot\bigg(\prod_{i=1}^{\ell}\big\|g_i\big\|_{(L\log L)^{p_i,\kappa}}\cdot
\prod_{j=\ell+1}^m\big\|g_j\big\|_{L^{p_j,\kappa}}\bigg).
\end{equation*}
This is our desired inequality, and the proof is complete.
\end{proof}

\section*{Acknowledgment}
The authors were supported by a grant from Xinjiang University under the project``Real-Variable Theory of Function Spaces and Its Applications".This work was supported by the Natural Science Foundation of China (No.XJEDU2020Y002 and 2022D01C407).

\end{document}